\pdfoutput=1
\documentclass[onefignum,onetabnum]{siamart220329}

\usepackage{bm}
\usepackage{enumitem}
\usepackage{multirow}

\theoremstyle{plain}
\newtheorem{theorem}{Theorem}[section]
\newtheorem{lemma}[theorem]{Lemma}
\newtheorem{proposition}[theorem]{Proposition}
\newtheorem{corollary}[theorem]{Corollary}
\theoremstyle{definition}
\newtheorem{assumption}[theorem]{Assumption}
\newtheorem{definition}[theorem]{Definition}
\theoremstyle{remark}
\newtheorem{remark}[theorem]{Remark}

\newcommand{\diag}{\operatorname{diag}}
\newcommand{\argmin}{\operatorname*{arg\,min}}
\newcommand{\Pih}{\Pi_h}

\title{Hessian-Recovery-Based \(C^0\) Finite Element Methods\\ for Non-Divergence Form Elliptic Equations}
\author{Minqiang Xu\(^{1}\), Boying Wu\(^{2}\), and Lei Zhang\(^{1}\)\\
\small \(^{1}\)School of Mathematical Sciences, Zhejiang University of Technology, Hangzhou 310023, PR China\\
\small \(^{2}\)Department of Mathematics, Harbin Institute of Technology, Harbin 150001, PR China}
\date{}
\headers{Hessian-Recovery-Based \(C^0\) FEM}{M. Xu, B. Wu, and L. Zhang}

\begin{document}
\maketitle

\begin{abstract}
A Hessian-recovery-based \(C^0\) finite element framework is proposed for second-order elliptic equations in non-divergence form. The construction is based on a direct approximation of the strong non-divergence operator: the Hessian \(D^2u\) is replaced by a recovered Hessian \(H_hu_h\), so that \(A:D^2u\) is approximated by \(A:H_hu_h\). The resulting discretizations include a nodal formulation and a Galerkin-type formulation for general Lagrange finite element spaces, as well as a biorthogonal Petrov--Galerkin formulation for linear elements. The analysis focuses on the recovered nodal matrix and identifies two verifiable algebraic solvability mechanisms. The first is a globally monotone regime leading to a discrete maximum principle, and the second is a localized Schur-complement criterion for sign-violating rows. A uniform inverse bound and a condition-number estimate are derived in the globally monotone case. Residual consistency estimates are obtained from the Hessian recovery error. In the globally monotone regime, these estimates combine with the uniform inverse bound to give a nodal \(L^\infty\)-error estimate for the nodal formulation. Numerical experiments with nonsmooth and discontinuous coefficients support the predicted algebraic diagnostics and show the accuracy of the proposed recovered-residual discretizations. A Monge--Amp\`ere type test further illustrates the use of the recovered Hessian in a Newton iteration for a fully nonlinear problem.
\end{abstract}



\section{Introduction}

A recurring difficulty in the numerical approximation of controlled diffusion and fully nonlinear problems is the treatment of linear elliptic operators in non-divergence form. Such operators appear, for example, in stochastic control and Hamilton--Jacobi--Bellman equations, as well as in Newton linearizations of fully nonlinear equations such as the Monge--Amp\`ere equation \cite{CaffarelliCabre1995,LakkisPryer2013,Oberman2008,FroeseOberman2011}. Their leading part is
\[
    A:D^2u=\sum_{\alpha,\beta=1}^2 a_{\alpha\beta}\partial_{\alpha\beta}u,
\]
rather than a divergence-form expression. As a result, the equation does not naturally fit into the standard \(H^1\)-variational framework. If the coefficient matrix \(A\) is nonsmooth or discontinuous, rewriting the operator in divergence form is generally unavailable or would introduce lower-order distributional terms. Hence the integration-by-parts arguments underlying classical finite element stability theory cannot be used directly. Numerical methods for such problems must therefore either approximate the strong operator itself or introduce alternative structures that play the role of a missing variational framework.

Existing methods realize these alternatives in several ways. One route relies on Cordes-type assumptions, under which Miranda--Talenti estimates recover a form of stability; this idea underlies a number of discontinuous Galerkin and \(C^0\) interior penalty methods, together with suitable stabilization \cite{SmearsSuli2013,FengHenningsNeilan2017,Kawecki2019}. A second route introduces a discrete substitute for the missing Hessian or enlarges the formulation, for instance through mixed finite element Hessians, least-squares formulations, weak Galerkin methods, or primal-dual weak Galerkin methods \cite{LakkisPryer2011,LakkisPryer2013,Gallistl2017,WangWang2018,LakkisMousavi2022}. Another line of work, especially for viscosity solutions and fully nonlinear equations, is based on monotone, stable, and consistent approximation principles and their finite-difference, wide-stencil, filtered, or discrete ABP-type realizations \cite{BarlesSouganidis1991,BonnansZidani2003,Oberman2008,FroeseOberman2011,NochettoZhang2018}. These approaches provide powerful and complementary tools, and they illustrate the variety of structures that can restore stability in the absence of a standard weak formulation.

A natural question is whether one can retain a standard \(C^0\) Lagrange finite element space while still approximating the strong non-divergence operator directly. The main obstacle is that a \(C^0\) finite element function does not provide a globally defined classical Hessian, so the strong operator cannot be evaluated in the usual pointwise sense. However, Lagrange finite element functions still retain local polynomial structure and nodal information. Hessian recovery uses this information to reconstruct second derivatives at the nodes and thereby provides a recovered Hessian for the discrete function. While recovery techniques are often used for derivative post-processing, superconvergence, and a posteriori error estimation, here the recovered derivatives are intended to enter the discrete operator itself. This makes it possible to build a direct approximation of the strong operator without using \(C^1\) elements or introducing auxiliary Hessian variables.

The interpretation of recovered derivatives as discrete differential operators has been explored in several \(C^0\) discretizations. For higher-order problems, including biharmonic equations, Cahn--Hilliard equations, and fourth-order regularizations of Monge--Amp\`ere equations, recovered Hessians provide the missing second-derivative information within \(C^0\) finite element spaces \cite{GuoZhangZou2018,XuGuoZou2019,ChenFengZhang2021,XuZhangWuLiu2025}. For non-divergence form equations, including Hamilton--Jacobi--Bellman problems, recovery ideas have also been combined with Cordes-type assumptions or discrete Miranda--Talenti estimates to obtain stable discretizations \cite{XuLinZou2023,ChuGuoZhang2025}. These developments establish recovered Hessians as viable substitutes for classical Hessians in \(C^0\)-based schemes. Existing recovery-based \(C^0\) discretizations, however, have been developed predominantly for linear or low-order elements.

Building on this viewpoint, the present work constructs a unified \(C^0\) nonvariational finite element framework around a recovered equation residual. Replacing \(D^2u\) by \(H_hu_h\) in \(A:D^2u=f\) gives the discrete residual \(A:H_hu_h-f\). The three schemes considered in this paper correspond to different ways of enforcing this residual. For general Lagrange finite element spaces, we formulate a nodal method and a Galerkin-type weak residual method. For linear elements, we further introduce a biorthogonal Petrov--Galerkin formulation that links the nodal residual with a variational testing procedure.

This recovered-residual framework shifts the main analytical question from variational stability to matrix-level solvability. The resulting matrices are neither generated by a coercive bilinear form nor automatically endowed with an \(M\)-matrix structure, so standard Galerkin stability arguments and classical monotone-matrix theory do not apply directly. We therefore analyze the recovered nodal matrix through its algebraic structure and formulate two verifiable solvability criteria. In the globally monotone regime, the constant-preserving property of the recovery operator yields a row-sum identity, which, together with the sign pattern of the sign-reversed matrix, leads to a discrete maximum-principle mechanism. When monotonicity is violated only locally, the sign-violating rows can be separated and the solvability test is reduced to a Schur-complement condition on the corresponding bad-row block. This gives a matrix-level criterion for checking solvability that is distinct from discrete Poincar\'e, Cordes, or Miranda--Talenti arguments.

The main contributions of this paper are summarized as follows.
\begin{enumerate}[leftmargin=2em]
\item We formulate a Hessian-recovery-based \(C^0\) nonvariational finite element framework for elliptic equations in non-divergence form. The framework includes a nodal recovered-residual method and a Galerkin-type recovered-residual method on general Lagrange finite element spaces, together with a biorthogonal Petrov--Galerkin realization of the nodal method for linear elements.
\item For the recovered nodal matrix, we develop two verifiable algebraic solvability criteria. The first applies in a globally monotone regime and yields a discrete maximum-principle mechanism; the second separates localized sign-violating rows and reduces the solvability test to a Schur-complement condition. These results also apply to the biorthogonal formulation through a positive diagonal row scaling. In the globally monotone case, we further derive a uniform inverse bound and an \(O(h^{-2})\) condition-number estimate.
\item We establish residual consistency estimates in terms of the Hessian recovery error. In the globally monotone regime, combining the nodal consistency estimate with the uniform inverse bound gives a nodal \(L^\infty\)-error estimate for the nodal formulation.
\item We conduct numerical experiments for nonsmooth coefficients, discontinuous coefficients, and a Monge--Amp\`ere type nonlinear problem. The experiments examine solution accuracy, condition-number behavior, and algebraic diagnostics, and they illustrate both the strengths and the different algebraic structures of the three formulations.
\end{enumerate}

Before turning to the detailed formulations and analysis, we specify the scope of the results proved below. The algebraic solvability theory is developed for the nodal recovered operator. Since the biorthogonal Petrov--Galerkin formulation differs from the nodal formulation only by a positive diagonal row scaling, the same nonsingularity results apply to that formulation as well. The Galerkin-type formulation, by contrast, provides a natural finite element realization of the same recovered residual, particularly for higher-order Lagrange spaces, but its mass-type averaging changes the nodal sign structure on which the maximum-principle and Schur-complement arguments rely. Therefore, the present analysis should not be interpreted as a stability theory for the Galerkin-type matrix. Rather, Scheme~2 is included as a computationally natural recovered-residual formulation, and its algebraic behavior is examined separately in the numerical experiments. Consistency estimates for the recovered residual are then established in the natural forms associated with the three formulations.

The rest of the paper is organized as follows. Section~\ref{sec:prelim} introduces the finite element spaces and the Hessian recovery operator. Section~\ref{sec:schemes} presents the recovered-residual formulations. Section~\ref{sec:solvability} develops the algebraic solvability criteria for the recovered nodal matrix. Section~\ref{sec:consistency} establishes residual consistency estimates and a nodal error consequence in the globally monotone regime. Section~\ref{sec:numerics} reports the numerical experiments, including nonsmooth and discontinuous coefficient tests and a Monge--Amp\`ere type nonlinear example. Section~\ref{sec:conclusion} concludes the paper.

\section{Preliminaries and the Hessian Recovery Operator}\label{sec:prelim}

\subsection{Model problem and finite element setting}

Let $\Omega\subset\mathbb R^2$ be a bounded polygonal domain. We consider the non-divergence form elliptic problem
\begin{equation}\label{eq:model}
    A:D^2u=f\quad\text{in }\Omega,
    \qquad
    u=g\quad\text{on }\partial\Omega,
\end{equation}
where $A=(a_{\alpha\beta})_{\alpha,\beta=1}^2$ is symmetric and uniformly positive definite. More precisely, there exist constants $0<\lambda_0\le \Lambda_0$ such that
\begin{equation}\label{eq:ellipticity}
    \lambda_0 |\xi|^2\le \xi^T A(x)\xi\le \Lambda_0 |\xi|^2,
    \qquad \forall \xi\in\mathbb R^2,
    \quad \hbox{a.e. }x\in\Omega.
\end{equation}
Since the nodal and biorthogonal formulations introduced below use nodal values of $A$ and $f$, we assume that these values are well defined, or that prescribed nodal approximations are used.

Let $\mathcal T_h$ be a shape-regular triangulation of $\Omega$ with mesh size $h$. Let
\[
    \mathcal N_h=\{z_i\}_{i=1}^N
\]
be the set of Lagrange nodes, and denote by $\mathcal N_h^I$ and $\mathcal N_h^B$ the sets of interior and boundary nodes, respectively. We define
\[
    I=\{i:z_i\in\mathcal N_h^I\},
    \qquad
    B=\{i:z_i\in\mathcal N_h^B\}.
\]
For an integer $k\ge1$, let
\begin{equation}\label{eq:Vh}
    V_h^k=
    \{v_h\in C^0(\overline\Omega): v_h|_T\in \mathbb P_k(T),\ \forall T\in\mathcal T_h\}.
\end{equation}
The homogeneous Dirichlet subspace is
\[
    V_{h,0}^k=
    \{v_h\in V_h^k: v_h(z_i)=0,\ z_i\in\mathcal N_h^B\}.
\]
For the nonhomogeneous boundary condition, we set
\[
    V_{h,g}^k=
    \{v_h\in V_h^k: v_h(z_i)=g(z_i),\ z_i\in\mathcal N_h^B\}.
\]
Let $\{\phi_i\}_{i=1}^N$ be the standard nodal basis of $V_h^k$.

\subsection{Polynomial-preserving Hessian recovery}

Standard $C^0$ Lagrange finite element functions do not have globally defined classical Hessians, because their gradients may be discontinuous across element interfaces. To approximate the non-divergence operator $A:D^2u$ within this finite element space, we employ a polynomial-preserving recovery (PPR) Hessian operator
\[
    H_h:V_h^k\to (V_h^k)^{2\times2}.
\]
The recovered strong operator used in the schemes below is $A:H_hu_h$. Although recovery operators are often used only as post-processing tools, the recovered Hessian here enters the discrete differential operator itself. We now define the PPR Hessian recovery operator and record the reproduction properties needed below.

For each node $z\in\mathcal N_h$, let $\omega_z$ be a local element patch around $z$, and choose a sampling set
\[
    \mathcal S_z\subset \mathcal N_h\cap\omega_z .
\]
Given $v_h\in V_h^k$, the PPR Hessian at $z$ is obtained by fitting a polynomial of degree $k+1$ to the nodal values of $v_h$ on $\mathcal S_z$. More precisely, let $p_z[v_h]\in\mathbb P_{k+1}(\omega_z)$ be the least-squares polynomial defined by
\begin{equation}\label{eq:ppr_ls}
    p_z[v_h]
    =
    \argmin_{p\in\mathbb P_{k+1}(\omega_z)}
    \sum_{\tilde z\in\mathcal S_z}
    |p(\tilde z)-v_h(\tilde z)|^2.
\end{equation}
The recovered Hessian is evaluated by differentiating the fitted polynomial at the node,
\begin{equation}\label{eq:ppr_hessian_node}
    (H_hv_h)(z)=D^2p_z[v_h](z),
\end{equation}
or, componentwise,
\begin{equation}\label{eq:ppr_component_node}
    H_h^{\alpha\beta}v_h(z)
    =
    \partial_{\alpha\beta}p_z[v_h](z),
    \qquad \alpha,\beta=1,2.
\end{equation}
The nodal values are then extended by Lagrange interpolation:
\begin{equation}\label{eq:hessian_interpolation}
    H_h^{\alpha\beta}v_h
    =
    \sum_{z_i\in\mathcal N_h}
    H_h^{\alpha\beta}v_h(z_i)\phi_i,
    \qquad \alpha,\beta=1,2.
\end{equation}
Thus $H_hv_h$ is a matrix-valued finite element function with entries in $V_h^k$.

The least-squares problem above is well posed only for sufficiently rich sampling sets. We impose this requirement through the following patch unisolvence condition.

\begin{assumption}[Patch unisolvence]\label{ass:ppr}
For every node $z\in\mathcal N_h$, the sampling set $\mathcal S_z$ is chosen so that the least-squares problem defining $p_z[v_h]$ has a unique solution in $\mathbb P_{k+1}(\omega_z)$.
\end{assumption}

Such sampling sets can be obtained by enlarging the local patch when necessary; see \cite{ZhangNaga2005,NagaZhang2005,GuoZhangZhao2017}. Under Assumption~\ref{ass:ppr}, the recovery operator has the following reproduction properties.

\begin{lemma}[Local polynomial reproduction]\label{lem:ppr_basic}
Under Assumption~\ref{ass:ppr}, the recovery operator $H_h$ is linear. Moreover, for each $z\in\mathcal N_h$, if $q\in\mathbb P_{k+1}(\omega_z)$ and $v_h\in V_h^k$ satisfies
\[
    v_h(\tilde z)=q(\tilde z),
    \qquad \forall \tilde z\in\mathcal S_z,
\]
then
\[
    H_hv_h(z)=D^2q(z).
\]
In particular, for any polynomial $q\in\mathbb P_{k+1}(\Omega)$,
\[
    H_h\Pih q(z)=D^2q(z),
    \qquad z\in\mathcal N_h.
\]
Consequently,
\begin{equation}\label{eq:constants_zero}
    H_hc=0,
    \qquad \forall c\in\mathbb R.
\end{equation}
\end{lemma}

\begin{proof}
For a fixed patch, the least-squares minimizer depends linearly on the sampled values of $v_h$, since the design matrix is fixed and the minimizer is unique by Assumption~\ref{ass:ppr}. Applying $D^2$ to the fitted polynomial gives the linearity of $H_h$.

If $v_h$ agrees with $q\in\mathbb P_{k+1}(\omega_z)$ on $\mathcal S_z$, then $p=q$ gives zero least-squares residual. By uniqueness of the minimizer, $p_z[v_h]=q$. Hence
\[
    H_hv_h(z)=D^2p_z[v_h](z)=D^2q(z).
\]
For a global polynomial $q\in\mathbb P_{k+1}(\Omega)$, the interpolant $\Pih q$ agrees with $q$ at all Lagrange nodes and hence on each sampling set $\mathcal S_z$. The reproduction identity therefore gives $H_h\Pih q(z)=D^2q(z)$. Taking $q$ to be constant yields $H_hc=0$.
\end{proof}

The identity $H_h1=0$, together with the partition of unity of the Lagrange basis, is the key ingredient in the row-sum identity for the recovered nodal matrices derived in Section~4.

\subsection{Recovery estimates}

The following estimate is the recovery result used in the subsequent consistency analysis. It is stated in a local form because the nodal formulation requires pointwise control at the Lagrange nodes, while the weak residual formulations use the corresponding $L^2$ bound.

\begin{lemma}[Local recovery estimates]\label{lem:hess_recovery_estimate}
Let the mesh family be shape-regular and quasi-uniform. Assume that the scaled local least-squares recovery functionals are uniformly stable in the sense that, for each node $z_i$,
\begin{equation}\label{eq:scaled_patch_stability}
    \left|H_h^{\alpha\beta}v_h(z_i)\right|
    \le
    Ch^{-2}
    \max_{\tilde z\in\mathcal S_{z_i}}
    |v_h(\tilde z)|,
    \qquad \alpha,\beta=1,2.
\end{equation}
Then, for $u\in W^{k+2,\infty}(\omega_{z_i})$,
\begin{equation}\label{eq:nodal_hessian_recovery}
    \left|
    D^2u(z_i)-H_h\Pih u(z_i)
    \right|
    \le
    Ch^k |u|_{W^{k+2,\infty}(\omega_{z_i})}.
\end{equation}
Consequently, if $u\in W^{k+2,\infty}(\Omega)$, then
\begin{equation}\label{eq:hess_recovery_estimate}
    \left\|D^2u-H_h\Pih u\right\|_{0,\Omega}
    \le
    Ch^k\|u\|_{W^{k+2,\infty}(\Omega)}.
\end{equation}
Here $\Pih u=I_h^ku$, and the constant $C$ is independent of $h$ and of the node $z_i$.
\end{lemma}

\begin{proof}
Let $q_i\in\mathbb P_{k+1}(\omega_{z_i})$ be the Taylor polynomial of $u$ at $z_i$ of degree $k+1$. By Lemma~\ref{lem:ppr_basic}, any finite element function whose nodal values agree with $q_i$ on $\mathcal S_{z_i}$ has recovered Hessian $D^2q_i(z_i)=D^2u(z_i)$ at $z_i$. Hence, by the linearity of the local recovery map at $z_i$,
\[
    H_h\Pih u(z_i)-D^2u(z_i)
    =
    H_h(\Pih u-v_i)(z_i),
\]
where $v_i\in V_h^k$ is any finite element function satisfying $v_i(\tilde z)=q_i(\tilde z)$ for all $\tilde z\in\mathcal S_{z_i}$. Since $\Pih u(\tilde z)=u(\tilde z)$ at every Lagrange node, Taylor's theorem gives
\[
    |\Pih u(\tilde z)-v_i(\tilde z)|
    =|u(\tilde z)-q_i(\tilde z)|
    \le
    Ch^{k+2}|u|_{W^{k+2,\infty}(\omega_{z_i})},
    \qquad \tilde z\in\mathcal S_{z_i}.
\]
The assumed scaled stability of the local recovery functional therefore implies
\[
\begin{aligned}
    \left|H_h(\Pih u-v_i)(z_i)\right|
    &\le
    Ch^{-2}
    \max_{\tilde z\in\mathcal S_{z_i}}
    |\Pih u(\tilde z)-v_i(\tilde z)|  \\
    &\le
    Ch^k |u|_{W^{k+2,\infty}(\omega_{z_i})}.
\end{aligned}
\]

It remains to derive the $L^2$ estimate. Let $I_h(D^2u)$ denote the componentwise Lagrange interpolation of $D^2u$. Since
\[
    I_h(D^2u)-H_h\Pih u
\]
is a matrix-valued finite element function, the stability of the Lagrange basis on shape-regular quasi-uniform meshes yields
\[
    \left\|
    I_h(D^2u)-H_h\Pih u
    \right\|_{0,\Omega}
    \le
    C
    \max_{z_i\in\mathcal N_h}
    \left|
    D^2u(z_i)-H_h\Pih u(z_i)
    \right|.
\]
Applying the nodal estimate gives
\[
    \left\|
    I_h(D^2u)-H_h\Pih u
    \right\|_{0,\Omega}
    \le
    Ch^k\|u\|_{W^{k+2,\infty}(\Omega)}.
\]
On the other hand, the standard interpolation estimate gives
\[
    \left\|
    D^2u-I_h(D^2u)
    \right\|_{0,\Omega}
    \le
    Ch^k\|u\|_{W^{k+2,\infty}(\Omega)}.
\]
Combining the last two estimates, we obtain
\[
    \left\|D^2u-H_h\Pih u\right\|_{0,\Omega}
    \le
    Ch^k\|u\|_{W^{k+2,\infty}(\Omega)}.
\]
This completes the proof.
\end{proof}

\begin{remark}\label{rem:scaled_patch_stability}
The stability condition in Lemma~\ref{lem:hess_recovery_estimate} is the standard scaled-patch stability requirement for polynomial-preserving recovery. It is satisfied, for example, on shape-regular quasi-uniform mesh families with fixed local recovery patterns, provided the corresponding scaled least-squares matrices are uniformly well conditioned. Under this condition, the $O(h^k)$ estimate follows from polynomial reproduction and local patch stability, in the same spirit as the Hessian recovery analysis of Guo--Zhang--Zhao~\cite{GuoZhangZhao2017}.
\end{remark}

With the recovered-Hessian operator $H_h$ defined above, the non-divergence operator is discretized by
\[
    A:D^2u
    \quad\leadsto\quad
    A:H_hu_h.
\]
This leads to the three Hessian-recovery-based nonvariational schemes introduced in the next section.

\section{Hessian-Recovery-Based Nonvariational Schemes}\label{sec:schemes}

Given the recovered Hessian operator \(H_h\), we use the recovered equation residual
\[
    r_h(u_h):=A:H_hu_h-f
\]
as the organizing quantity for the discretizations. Two realizations are first formulated on general Lagrange finite element spaces: a nodal formulation, which enforces the residual at the interior nodes, and a Galerkin-type formulation, which tests the same residual against the standard finite element basis functions. In the linear-element setting, the nodal residual further admits a biorthogonal Petrov--Galerkin realization, linking pointwise enforcement with a variational testing procedure. All three formulations are therefore built from the same recovered Hessian construction, while their different enforcement mechanisms lead to different algebraic structures.

Throughout this section, if
\[
    u_h=\sum_{j=1}^N U_j\phi_j,
\]
then \(U=(U_1,\ldots,U_N)^T\) denotes the vector of nodal values, with interior and boundary parts \(U_I\) and \(U_B\).

\subsection{Scheme 1: Nodal formulation}

The first realization is the nodal formulation of the recovered residual. It is useful to separate two algebraic ingredients: the matrices representing the recovered Hessian itself, and the coefficient-weighted matrix representing the recovered non-divergence operator.

For \(\alpha,\beta=1,2\), define
\begin{equation}\label{eq:Rmatrix}
    R^{\alpha\beta}_{ij}
    =\bigl(H_h^{\alpha\beta}\phi_j\bigr)(z_i),
    \qquad i,j=1,\ldots,N .
\end{equation}
For a finite element function
\[
    u_h=\sum_{j=1}^N U_j\phi_j,
\]
the linearity of \(H_h\) gives
\begin{equation}\label{eq:R_action}
    \bigl(R^{\alpha\beta}U\bigr)_i
    =H_h^{\alpha\beta}u_h(z_i).
\end{equation}
Thus \(R^{\alpha\beta}\) maps nodal values to the nodal values of the recovered \((\alpha,\beta)\)-Hessian component.

The coefficient matrix \(A\) is then incorporated row by row. For \(A=(a_{\alpha\beta})_{\alpha,\beta=1}^2\), define
\begin{equation}\label{eq:KAdef}
    (K_A)_{ij}
    =\sum_{\alpha,\beta=1}^2
    a_{\alpha\beta}(z_i)R^{\alpha\beta}_{ij},
    \qquad i,j=1,\ldots,N .
\end{equation}
Equivalently,
\begin{equation}\label{eq:KAsum}
    K_A=\sum_{\alpha,\beta=1}^2 D_{\alpha\beta}R^{\alpha\beta},
    \qquad
    D_{\alpha\beta}=
    \diag\bigl(a_{\alpha\beta}(z_1),\ldots,a_{\alpha\beta}(z_N)\bigr).
\end{equation}
It follows that
\begin{equation}\label{eq:KU_residual}
    (K_AU)_i=A(z_i):H_hu_h(z_i),
\end{equation}
so \(K_A\) represents the nodal action of the recovered non-divergence operator.

\paragraph{Discrete formulation.}
Find \(u_h\in V_{h,g}^k\) such that
\begin{equation}\label{eq:scheme1}
    A(z_i):H_hu_h(z_i)=f(z_i),
    \qquad z_i\in\mathcal N_h^I .
\end{equation}
The equations are imposed at the interior nodes, while the discrete solution is sought in the conforming \(C^0\) Lagrange finite element space. Hence Scheme~1 is a nodal nonvariational formulation within the standard finite element trial space.

\paragraph{Matrix form.}
Let
\begin{equation}\label{eq:scheme1_K1_F1}
    K_{1,h}:=(K_A)_{II},
    \qquad
    (F_1)_i=f(z_i),\quad i\in I .
\end{equation}
Eliminating the prescribed boundary degrees of freedom gives
\begin{equation}\label{eq:scheme1_matrix}
    K_{1,h}U_I
    =F_{1,I}-(K_A)_{IB}U_B .
\end{equation}
Here \(U_I\) is the vector of interior unknowns, and \(U_B\) is fixed by the prescribed boundary data. The matrix \(K_{1,h}\) retains the nodal structure of the recovered operator \(A:H_h\) and is the matrix analyzed in the algebraic solvability results of Section~\ref{sec:solvability}.

\subsection{Scheme 2: Galerkin-type Hessian-recovery-based NFEM}

The second scheme is a Galerkin-type formulation of the recovered non-divergence residual. Rather than enforcing the residual at nodes, it tests $A:H_hu_h$ against standard finite element test functions.

\paragraph{Discrete formulation.}
Find $u_h\in V_{h,g}^k$ such that
\begin{equation}\label{eq:scheme2}
    (A:H_hu_h,v_h)=(f,v_h),
    \qquad \forall v_h\in V_{h,0}^k .
\end{equation}
This formulation is a natural finite element realization of the recovered residual. It keeps both the trial and test functions in standard \(C^0\) Lagrange finite element spaces and is particularly convenient for higher-order elements, for which the recovered Hessian construction and the residual testing can be carried out without introducing additional Hessian unknowns.

\paragraph{Matrix form.}
For $\alpha,\beta=1,2$, define the weighted mass matrices
\begin{equation}\label{eq:M_A}
    (M_A^{\alpha\beta})_{\ell i}
    =\int_\Omega a_{\alpha\beta}(x)\phi_i(x)\phi_\ell(x)\,dx .
\end{equation}
Then the full matrix associated with Scheme~2 is
\begin{equation}\label{eq:CAmatrix}
    C_A=\sum_{\alpha,\beta=1}^2 M_A^{\alpha\beta}R^{\alpha\beta}.
\end{equation}
Let
\begin{equation}\label{eq:scheme2_K2_F2}
    K_{2,h}:=(C_A)_{II},
    \qquad
    (F_2)_\ell=(f,\phi_\ell),\quad \ell\in I .
\end{equation}
After separating interior and boundary degrees of freedom, Scheme~2
becomes
\begin{equation}\label{eq:scheme2_matrix}
    K_{2,h}U_I
    =F_{2,I}-(C_A)_{IB}U_B .
\end{equation}
Compared with Scheme~1, Scheme~2 has a different algebraic structure. Its matrix contains the coefficient-weighted mass matrices $M_A^{\alpha\beta}$, which average the nodal recovered residual against standard finite element test functions. This averaging generally mixes the nodal rows of $K_A$ and therefore does not preserve the off-diagonal sign structure used in the maximum-principle and Schur-complement analysis of Section~\ref{sec:solvability}. Thus Scheme~2 should be viewed as a Galerkin-type recovered-residual formulation rather than as a matrix-level monotone discretization. In particular, the solvability arguments developed below for the nodal matrix do not apply directly to $K_{2,h}$. Scheme~2 is nevertheless retained because it is a standard finite element realization of the same recovered residual and is convenient for higher-order Lagrange spaces. Its algebraic behavior is assessed separately in Section~\ref{sec:numerics}.

\subsection{Scheme 3: Biorthogonal Petrov--Galerkin formulation for linear elements}

The third formulation is defined for linear Lagrange elements. In this case, the nodal residual admits a Petrov--Galerkin interpretation through a biorthogonal test space.

We construct the test functions locally. On each element \(T\in\mathcal T_h\), let \(\{\lambda_a^T\}_{a=1}^3\) be the barycentric basis of \(\mathbb P_1(T)\). We choose the local weights \(\mu_a^T=|T|/3\) and define
\begin{equation}\label{eq:P1_dual}
    \eta_a^T=4\lambda_a^T-1,
    \qquad a=1,2,3.
\end{equation}
Then
\begin{equation}\label{eq:local_biorthogonal}
    \int_T \lambda_c^T\eta_a^T\,dx=\mu_a^T\delta_{ac}.
\end{equation}
For a global vertex \(z_i\in\mathcal N_h\), define the corresponding global dual basis function \(\psi_i\) elementwise by
\begin{equation}\label{eq:global_dual}
    \psi_i|_T=
    \begin{cases}
    \eta_a^T,
    & \text{if } z_i \text{ corresponds to the local vertex } a \text{ of }T,\\
    0,
    & \text{if } z_i\notin T .
    \end{cases}
\end{equation}
Set
\begin{equation}\label{eq:Wh}
    W_h^1=\operatorname{span}\{\psi_i:z_i\in\mathcal N_h\}.
\end{equation}
The space \(W_h^1\) is generally discontinuous and is locally linear on each element. The construction yields the global biorthogonality relation
\begin{equation}\label{eq:global_biorthogonal}
    (\phi_i,\psi_j)=\mu_j\delta_{ij},
    \qquad
    \mu_j=\sum_{T\ni z_j}\mu_{a(j,T)}^T>0 .
\end{equation}

\paragraph{Discrete formulation.}

Let
\[
    W_{h,0}^1=\operatorname{span}\{\psi_i:z_i\in\mathcal N_h^I\}.
\]
Find $u_h\in V_{h,g}^1$ such that
\begin{equation}\label{eq:scheme3}
    (I_h^1(A:H_hu_h),w_h)=(f,w_h),
    \qquad \forall w_h\in W_{h,0}^1,
\end{equation}
where
\begin{equation}\label{eq:nodal_residual_interp}
    I_h^1(A:H_hu_h)
    =\sum_{i=1}^N \big(A(z_i):H_hu_h(z_i)\big)\phi_i
\end{equation}
is the Lagrange interpolation of the recovered residual. This
formulation is genuinely Petrov--Galerkin: the trial space is the
standard continuous finite element space $V_{h,g}^1$, while the test
space is the biorthogonal space $W_{h,0}^1$.

\paragraph{Matrix form.}
Taking $w_h=\psi_j$ in \eqref{eq:scheme3} and using the biorthogonality
relation \eqref{eq:global_biorthogonal} gives
\begin{equation}\label{eq:scheme3_node}
    \mu_j A(z_j):H_hu_h(z_j)=(f,\psi_j),
    \qquad z_j\in\mathcal N_h^I .
\end{equation}
Therefore Scheme~3 has the matrix form
\begin{equation}\label{eq:scheme3_matrix}
    K_{3,h}U_I
    =F_{3,I}-D_\mu(K_A)_{IB}U_B,
\end{equation}
where
\begin{equation}\label{eq:Dmu}
    K_{3,h}:=D_\mu K_{1,h},
    \qquad
    D_\mu=\diag(\mu_j)_{j\in I},
    \qquad
    (F_3)_j=(f,\psi_j),\quad j\in I .
\end{equation}
Thus Scheme~3 is a Petrov--Galerkin realization of the nodal scheme. Its
left-hand matrix is a positive diagonal row scaling of the Scheme~1
matrix. Consequently, Scheme~3 has the same sign-violation pattern as
Scheme~1 and inherits algebraic properties that are invariant under
positive diagonal row scaling.

\begin{remark}[Computation of the right-hand side]
No dual mesh is needed in the implementation. For the linear-element construction,
\begin{equation}\label{eq:rhs_dual}
    (f,\psi_j)=\sum_{T\ni z_j}\int_T f\,\eta_{a(j,T)}^T\,dx .
\end{equation}
Thus the right-hand side is assembled by standard element quadrature.
Equivalently, one may first compute the standard local load vector
\[
    F_b^T=\int_T f\ell_b^T\,dx,
\]
and then form
\[
    \widetilde F_a^T=\mu_a^T\sum_b(M_T^{-1})_{ba}F_b^T .
\]
\end{remark}

\begin{remark}[Relation among the three schemes]
The three schemes are built from the same recovered residual $A:H_hu_h-f$, but they enforce it in different ways. Scheme~1 enforces the residual at the interior nodes and leads to the nodal matrix $K_{1,h}$. Scheme~2 tests the recovered residual against standard finite element functions and leads to the Galerkin-type matrix $K_{2,h}$. Scheme~3 tests the interpolated residual against the biorthogonal test space and satisfies
\[
    K_{3,h}=D_\mu K_{1,h}.
\]
Thus Scheme~3 is a Petrov--Galerkin realization of the nodal scheme, up to a positive diagonal row scaling. This identity explains why the algebraic solvability theory in Section~\ref{sec:solvability} is developed for $K_{1,h}$: the corresponding nonsingularity results immediately carry over to $K_{3,h}$. The Galerkin-type matrix $K_{2,h}$, by contrast, contains mass-type averaging and therefore has a different sign structure from the nodal matrix. It is consequently treated as a natural recovered-residual finite element formulation whose algebraic behavior is assessed separately in the numerical experiments.
\end{remark}

\section{Algebraic Structure, Unique Solvability, and Conditioning}
\label{sec:solvability}

The three schemes introduced above are all constructed from the recovered residual $A:H_hu_h-f$, but they lead to different matrix structures. In this section we analyze the matrix that preserves the pointwise nodal structure of the recovered operator, namely the Scheme~1 matrix
\[
    K_{1,h}=(K_A)_{II}.
\]
The biorthogonal formulation satisfies $K_{3,h}=D_\mu K_{1,h}$ with a positive diagonal matrix $D_\mu$, and therefore has the same nonsingularity properties as $K_{1,h}$. By contrast, the Galerkin-type matrix $K_{2,h}$ contains the mass-type averaging induced by standard finite element testing and does not preserve the nodal sign structure used below.

For Hessian-recovery-based $C^0$ nonvariational discretizations, a standard coercive variational framework is not directly available, and the recovered nodal matrix need not inherit the global monotone structure of classical finite difference matrices. We therefore seek verifiable matrix-level criteria for nonsingularity. The analysis identifies two such criteria: a globally monotone regime based on a discrete maximum principle, and a localized sign-violation regime based on a Schur-complement test.

The globally monotone regime should be understood as a clean baseline case in which the recovered nodal operator retains an $M$-matrix type sign structure. This situation is observed on the uniform meshes in the numerical tests below. On general unstructured meshes, or even under mild perturbations of a uniform mesh, the recovered stencil may lose the global off-diagonal sign property while still preserving the row-sum identity. The localized Schur-complement criterion is designed for this latter situation: it does not require the full nodal matrix to be an $M$-matrix, but instead isolates the rows where the sign condition fails.

Throughout this section, we use the sign-reversed matrices $L_A=-K_A$ and $L_h=-K_{1,h}$. This convention corresponds to the elliptic operator $-A:D^2$ and is the natural sign convention for maximum-principle and $M$-matrix arguments.

\subsection{Row-sum identity}

We first record the matrix consequence of the constant-preserving
property of the Hessian recovery operator. The identity $H_h1=0$ was
proved in Lemma~\ref{lem:ppr_basic}; here we only translate it into the
row-sum structure of the recovered Hessian matrices.

\begin{lemma}[Row-sum identity]
\label{lem:row_sum_R}
For each $\alpha,\beta=1,2$,
\begin{equation}\label{eq:R_rowsum}
    R^{\alpha\beta}\mathbf 1=0.
\end{equation}
Consequently,
\begin{equation}\label{eq:KA_rowsum}
    K_A\mathbf 1=0,
    \qquad
    L_A\mathbf 1=0.
\end{equation}
\end{lemma}

\begin{proof}
By the partition of unity property of the Lagrange basis,
$\sum_{j=1}^N\phi_j\equiv1$. Therefore, for every node $z_i$,
\[
    \sum_{j=1}^N R_{ij}^{\alpha\beta}
    =
    H_h^{\alpha\beta}\left(\sum_{j=1}^N\phi_j\right)(z_i)
    =
    H_h^{\alpha\beta}1(z_i).
\]
By Lemma~\ref{lem:ppr_basic}, $H_h1=0$. Hence
$R^{\alpha\beta}\mathbf 1=0$. Since
\[
    K_A=
    \sum_{\alpha,\beta=1}^2D_{\alpha\beta}R^{\alpha\beta},
\]
we obtain $K_A\mathbf 1=0$. Finally, $L_A=-K_A$ gives
$L_A\mathbf 1=0$.
\end{proof}

\begin{remark}
The row-sum identity is inherited from polynomial reproduction of the
recovery operator. It is independent of the coefficient values
$a_{\alpha\beta}(z_i)$ and should be distinguished from monotonicity.
The latter is a sign property of the matrix entries and depends on the
coefficient matrix, the mesh, and the recovery stencils.
\end{remark}

\subsection{Global monotonicity and unique solvability}

Motivated by the $M$-matrix structure of monotone discretizations of
$-\Delta$, we first identify a clean algebraic regime in which the
row-sum identity leads to a discrete maximum principle.

\begin{definition}[Globally monotone nodal matrix]
\label{def:global_monotone}
We say that the recovered nodal matrix is globally monotone if the
sign-reversed matrix $L_A=-K_A$ satisfies
\begin{equation}\label{eq:global_monotone_sign}
    (L_A)_{ij}\le0,
    \qquad
    i\in I,\quad j\in I\cup B,\quad j\ne i,
\end{equation}
and if its negative graph is connected to the Dirichlet boundary: for
every nonempty subset $S\subset I$, there exist $i\in S$ and
$j\in (I\setminus S)\cup B$ such that
\begin{equation}\label{eq:global_monotone_connectivity}
    (L_A)_{ij}<0.
\end{equation}
\end{definition}

The first condition is the $Z$-matrix sign condition for the positive
operator $L_A$. The second one rules out interior components disconnected
from the prescribed boundary data.

\begin{theorem}[Solvability criterion under global monotonicity]
\label{thm:unique_global_monotone}
If the recovered nodal matrix is globally monotone, then Scheme~1 and
Scheme~3 are uniquely solvable, and $L_h=-K_{1,h}$ is a nonsingular
$M$-matrix.
\end{theorem}

\begin{proof}
It suffices to prove a discrete maximum principle for $L_h$. Let $U\in\mathbb R^{|I|}$ and assume that $L_hU\ge0$. Extend $U$ to all nodes by setting $U_b=0$ for $b\in B$. Suppose, to the contrary, that $U$ attains a negative minimum at an interior node. Set
\[
    m=\min_{i\in I}U_i<0,
    \qquad
    S=\{i\in I:U_i=m\}.
\]
For $i\in S$, the full row-sum identity $L_A\mathbf1=0$ gives
\begin{equation}\label{eq:Lh_dmp_identity}
    (L_hU)_i
    =
    \sum_{\substack{j\in I\\ j\ne i}}
    -(L_A)_{ij}(U_i-U_j)
    +
    \sum_{b\in B}-(L_A)_{ib}U_i .
\end{equation}
Since $U_i\le U_j$ for all $j\in I$, $U_i<0$, and $(L_A)_{ij}\le0$ for all $j\ne i$, every term on the right-hand side is nonpositive. Hence $(L_hU)_i\le0$. Together with $L_hU\ge0$, this implies $(L_hU)_i=0$ for all $i\in S$, and therefore all the nonpositive terms above must vanish.

On the other hand, the boundary connectivity condition applied to the nonempty set $S$ yields $i\in S$ and $j\in (I\setminus S)\cup B$ such that $(L_A)_{ij}<0$. If $j\in I\setminus S$, then $U_j>U_i$, so the corresponding term
\[
    -(L_A)_{ij}(U_i-U_j)
\]
is strictly negative. If $j\in B$, then $U_j=0$ and $U_i<0$, so the boundary contribution
\[
    -(L_A)_{ij}U_i
\]
is strictly negative. In both cases we obtain a contradiction to the vanishing of all terms. Therefore no negative interior minimum is possible, and hence $U\ge0$.

Thus
\[
    L_hU\ge0\quad\Longrightarrow\quad U\ge0 .
\]
If $L_hU=0$, applying this implication to both $U$ and $-U$ gives $U=0$. Hence $L_h$ is nonsingular. Moreover, for any $r\ge0$, the solution $U=L_h^{-1}r$ satisfies $U\ge0$, so $L_h^{-1}\ge0$. Since $L_h$ is a $Z$-matrix, it is a nonsingular $M$-matrix.

Finally, $K_{1,h}=-L_h$, so Scheme~1 is uniquely solvable. Since $K_{3,h}=D_\mu K_{1,h}$ with $D_\mu$ positive diagonal, $K_{3,h}$ is nonsingular if and only if $K_{1,h}$ is nonsingular. Therefore Scheme~3 is also uniquely solvable.
\end{proof}

\subsection{Localized sign violations and a Schur-complement criterion}

The globally monotone regime gives a clean sufficient criterion, but it can be too restrictive on general unstructured meshes. In this setting
the recovered nodal matrix may fail to be a global $M$-matrix. In many
cases, however, the violations of the off-diagonal sign condition are
confined to a relatively small set of rows. This motivates a localized Schur-complement criterion: the sign-violating rows are separated from the monotone part of the matrix, and nonsingularity is tested on the resulting reduced matrix.

Define
\begin{equation}\label{eq:bad_good_sets}
    I_{\rm bad}
    =
    \{i\in I:\exists j\in I\cup B,
    \ j\ne i \text{ such that }(L_A)_{ij}>0\},
    \qquad
    I_{\rm good}=I\setminus I_{\rm bad}.
\end{equation}
For brevity, write $S=I_{\rm bad}$ and $G=I_{\rm good}$. After
reordering the unknowns according to $I=G\cup S$, we write
\begin{equation}\label{eq:block_L}
    L_h=
    \begin{pmatrix}
    L_{GG} & L_{GS}\\
    L_{SG} & L_{SS}
    \end{pmatrix},
\end{equation}
where the subscripts indicate the corresponding row and column index
sets.

The good rows satisfy the desired off-diagonal sign condition. We say
that the good block is connected to the exterior if for every nonempty
subset $T\subset G$, there exist $i\in T$ and
$j\in (G\setminus T)\cup S\cup B$ such that
\begin{equation}\label{eq:good_block_connectivity}
    (L_A)_{ij}<0.
\end{equation}
Here the bad rows and the Dirichlet boundary are both regarded as
exterior nodes for the good block.

\begin{lemma}[Nonsingularity of the good block]
\label{lem:good_block_nonsing}
If the good block is connected to the exterior, then $L_{GG}$ is a
nonsingular $M$-matrix.
\end{lemma}

\begin{proof}
We prove a discrete maximum principle for $L_{GG}$. Let $U\in\mathbb R^{|G|}$ satisfy $L_{GG}U\ge0$. Extend $U$ by zero on $S\cup B$. Suppose that $U$ attains a negative minimum on $G$. Set
\[
    m=\min_{i\in G}U_i<0,
    \qquad
    T=\{i\in G:U_i=m\}.
\]
For $i\in T$, the row-sum identity for the full matrix $L_A$, together with the good-row sign condition, gives
\begin{equation}\label{eq:good_block_dmp_identity}
    (L_{GG}U)_i
    =
    \sum_{\substack{j\in G\\ j\ne i}}
    -(L_A)_{ij}(U_i-U_j)
    +
    \sum_{j\in S\cup B}-(L_A)_{ij}U_i
    \le0.
\end{equation}
Since $L_{GG}U\ge0$, all terms above must vanish for every $i\in T$.

The exterior connectivity condition gives $i\in T$ and $j\in (G\setminus T)\cup S\cup B$ such that $(L_A)_{ij}<0$. If $j\in G\setminus T$, then $U_j>U_i$, and the corresponding interior term is strictly negative. If $j\in S\cup B$, then the extended value is zero, while $U_i<0$, and the exterior contribution is strictly negative. Both cases contradict the vanishing of all terms. Hence no negative minimum exists, and $U\ge0$.

Therefore
\[
    L_{GG}U\ge0\quad\Longrightarrow\quad U\ge0 .
\]
Applying this implication to both $U$ and $-U$ when $L_{GG}U=0$ gives $U=0$, so $L_{GG}$ is nonsingular. Moreover, $L_{GG}^{-1}\ge0$. Since $L_{GG}$ is a $Z$-matrix, it is a nonsingular $M$-matrix.
\end{proof}

Since $L_{GG}$ is nonsingular under this condition, the Schur complement
associated with the sign-violating rows is well defined:
\begin{equation}\label{eq:schur}
    \mathfrak S_h
    =
    L_{SS}-L_{SG}L_{GG}^{-1}L_{GS}.
\end{equation}

\begin{theorem}[Localized Schur-complement solvability criterion]
\label{thm:schur_unique}
Assume that the good block is connected to the exterior. If the Schur
complement $\mathfrak S_h$ is nonsingular, then Scheme~1 and Scheme~3
are uniquely solvable.
\end{theorem}

\begin{proof}
By Lemma~\ref{lem:good_block_nonsing}, $L_{GG}$ is a nonsingular $M$-matrix and, in particular, is invertible. Hence the Schur complement $\mathfrak S_h$ is well defined. With respect to the decomposition $I=G\cup S$, the block matrix $L_h$ admits the factorization
\[
L_h
=
\begin{pmatrix}
L_{GG} & L_{GS}\\
L_{SG} & L_{SS}
\end{pmatrix}
=
\begin{pmatrix}
I & 0\\
L_{SG}L_{GG}^{-1} & I
\end{pmatrix}
\begin{pmatrix}
L_{GG} & L_{GS}\\
0 & \mathfrak S_h
\end{pmatrix}.
\]
Both factors on the right-hand side are nonsingular if $L_{GG}$ and $\mathfrak S_h$ are nonsingular. Therefore $L_h$ is nonsingular. Since $K_{1,h}=-L_h$, Scheme~1 is uniquely solvable. Finally, $K_{3,h}=D_\mu K_{1,h}$ with $D_\mu$ positive diagonal, so $K_{3,h}$ is nonsingular if and only if $K_{1,h}$ is nonsingular. Thus Scheme~3 is uniquely solvable.
\end{proof}

\begin{remark}\label{rem:schur_criterion}
The Schur-complement criterion provides a localized algebraic solvability test. When the global off-diagonal sign condition fails only in a subset of rows, the monotone part of the matrix is retained in the good block $L_{GG}$, while the influence of the sign-violating rows is represented by the reduced matrix $\mathfrak S_h$. Thus the full nodal matrix need not be an $M$-matrix; the test instead determines whether the localized sign violations destroy nonsingularity after the good degrees of freedom have been eliminated.

It is important to note that this criterion is an assembled-matrix, a posteriori algebraic certificate. It is not intended to provide a purely geometric mesh condition or a coefficient-level sufficient condition for monotonicity. Rather, once the recovered nodal matrix has been assembled, the criterion checks whether the localized sign violations have destroyed the invertibility of the operator after the good degrees of freedom have been eliminated.

This interpretation is consistent with a standard perturbation viewpoint. Indeed,
\[
    \mathfrak S_h=L_{SS}-L_{SG}L_{GG}^{-1}L_{GS}
\]
can be regarded as the bad-row block $L_{SS}$ perturbed by the coupling through the monotone block $L_{GG}$. In particular, if $L_{SS}$ is nonsingular and, in a subordinate matrix norm,
\[
    \bigl\|L_{SS}^{-1}L_{SG}L_{GG}^{-1}L_{GS}\bigr\|<1,
\]
then $\mathfrak S_h$ is nonsingular by the Neumann-series argument. This condition is not imposed as a diagnostic requirement below; rather, it illustrates one perturbative regime in which localized sign violations do not destroy nonsingularity. The numerical diagnostics therefore focus on the direct Schur-complement quantities reported for $\mathfrak S_h$.
\end{remark}

\subsection{Uniform inverse bound and condition-number estimate under global monotonicity}

We next derive a uniform inverse bound and a condition-number estimate in
the globally monotone regime. These two estimates use different parts of
the algebraic structure. The inverse bound follows from the discrete
maximum principle together with a quadratic barrier argument. The
\(O(h^{-2})\) condition-number estimate additionally uses a stencil-scaling
bound for the recovered Hessian on uniform meshes with fixed recovery
patches.

For the stencil-scaling estimate, we consider a family of uniform meshes
with fixed local recovery patterns. In this setting the scaled
least-squares systems are identical up to translation, rotation, and a
finite number of boundary patch types. Hence the recovery stencils are
uniformly stable, and the following bound follows from scaling.

\begin{lemma}[Stencil scaling on uniform meshes]
\label{lem:L_norm_bound}
On a family of uniform meshes with fixed local recovery patches, there
exists a constant $C$, independent of $h$, such that
\begin{equation}\label{eq:ppr_stencil_bound}
    \sum_{j=1}^N |R_{ij}^{\alpha\beta}|
    \le Ch^{-2},
    \qquad
    i=1,\ldots,N,
    \quad
    \alpha,\beta=1,2.
\end{equation}
Consequently, if $A\in [L^\infty(\Omega)]^{2\times2}$, then
\begin{equation}\label{eq:L_norm_bound}
    \|L_h\|_\infty\le Ch^{-2}.
\end{equation}
\end{lemma}

\begin{proof}
For each fixed patch type, the recovered value
$(H_h^{\alpha\beta}v_h)(z_i)$ is a linear functional of finitely many
sampled nodal values. After scaling the patch to unit size, the
coefficients of this functional are bounded independently of $h$. Scaling
back to the physical patch introduces the factor $h^{-2}$, because a
second derivative is being recovered. Since only finitely many patch
types occur on the uniform mesh family, the stencil bound follows.

The bound for $L_h$ follows from
\[
    (K_A)_{ij}
    =
    \sum_{\alpha,\beta=1}^2
    a_{\alpha\beta}(z_i)R^{\alpha\beta}_{ij}
\]
and the boundedness of the coefficients. Since $L_h=-K_{1,h}$, the same
bound holds for $L_h$.
\end{proof}

\begin{theorem}[Uniform inverse bound and condition-number estimate under global monotonicity]
\label{thm:cond_monotone}
Assume that the recovered nodal matrix is globally monotone. Then the
sign-reversed nodal matrix $L_h=-K_{1,h}$ satisfies the uniform inverse
bound
\begin{equation}\label{eq:inverse_bound_monotone}
    \|L_h^{-1}\|_\infty\le C,
\end{equation}
where $C$ is independent of $h$. In the uniform-mesh setting with fixed
recovery patches, the stencil-scaling estimate further gives
\begin{equation}\label{eq:L_norm_monotone}
    \|L_h\|_\infty\le Ch^{-2},
\end{equation}
and therefore
\begin{equation}\label{eq:cond_monotone}
    \kappa_\infty(L_h)\le Ch^{-2}.
\end{equation}
\end{theorem}

\begin{proof}
We first prove the inverse bound, which uses only the globally monotone
algebraic structure. Choose $x_0\in\mathbb R^2$ and $R>0$, independent of
$h$, such that $\Omega\subset B_R(x_0)$, and set
\[
    q(x)=R^2-|x-x_0|^2.
\]
Then $q\ge0$ on $\overline\Omega$, and $\|q\|_{L^\infty(\Omega)}$ is
bounded independently of $h$. Since $q\in\mathbb P_2$ and $k\ge1$, the
polynomial reproduction property gives
\[
    H_h\Pi_hq=D^2q=-2I.
\]
Therefore, at every interior node $z_i$,
\[
    (L_A\Pi_hq)(z_i)
    =
    -A(z_i):H_h\Pi_hq(z_i)
    =
    2\,\operatorname{tr}A(z_i)
    \ge c_0,
\]
where $c_0>0$ depends only on the ellipticity constant.

Let $q_I$ and $q_B$ denote the interior and boundary nodal values of
$\Pi_hq$. Since the boundary off-diagonal entries of $L_A$ are nonpositive
in the globally monotone regime and $q_B\ge0$, we have
\[
    L_hq_I
    =
    (L_A\Pi_hq)_I-(L_A)_{IB}q_B
    \ge c_0\mathbf 1.
\]
By Theorem~\ref{thm:unique_global_monotone}, $L_h^{-1}\ge0$. Hence
\[
    L_h^{-1}\mathbf 1\le c_0^{-1}q_I .
\]
For any $r\in\mathbb R^{|I|}$, let $U=L_h^{-1}r$. Then
\[
    |U|
    \le
    L_h^{-1}|r|
    \le
    \|r\|_\infty L_h^{-1}\mathbf 1
    \le
    c_0^{-1}\|r\|_\infty q_I .
\]
Since $q_I$ is uniformly bounded, it follows that
$\|U\|_\infty\le C\|r\|_\infty$. Thus
$\|L_h^{-1}\|_\infty\le C$.

In the uniform-mesh setting with fixed recovery patches,
Lemma~\ref{lem:L_norm_bound} gives $\|L_h\|_\infty\le Ch^{-2}$.
Combining this estimate with the inverse bound yields
\[
    \kappa_\infty(L_h)
    =
    \|L_h\|_\infty\|L_h^{-1}\|_\infty
    \le Ch^{-2}.
\]
\end{proof}

\begin{remark}
The uniform inverse bound in Theorem~\ref{thm:cond_monotone} follows from
the globally monotone algebraic structure and the quadratic barrier
argument. The uniform-mesh assumption is used for the stencil-scaling
estimate of $\|L_h\|_\infty$, and hence for the condition-number bound.
The same condition-number scaling applies to the biorthogonal matrix
$D_\mu L_h$, provided the weights $\mu_j$ are mutually comparable on the
mesh family. On unstructured meshes with localized sign violations, the
estimate above is not proved here. The numerical diagnostics in
Section~\ref{sec:numerics} nevertheless indicate similar
$O(h^{-2})$-type growth in the tested cases when the sign violations
remain localized and the associated Schur complements remain nonsingular.
\end{remark}

\section{Consistency of the Recovered Residual}\label{sec:consistency}

The preceding section establishes matrix-level solvability mechanisms for the recovered nodal operator. These results are based on the sign structure of the nodal matrix and are therefore distinct from standard coercive Galerkin stability or inf-sup arguments. We now record the consistency part of the recovered-residual framework.

Let $u$ be the exact solution of
\[
    A:D^2u=f,
\]
and let $\Pih u=I_h^ku$ be its Lagrange interpolant. The recovered residual obtained by inserting the interpolated exact solution into the discrete operator is
\[
    r_h(u):=A:H_h\Pih u-f.
\]
Since $f=A:D^2u$, we have
\[
    r_h(u)=A:(H_h\Pih u-D^2u).
\]
Thus the recovered residual is controlled directly by the Hessian recovery error. Using Lemma~\ref{lem:hess_recovery_estimate}, we obtain the following basic estimate.

\begin{lemma}[Recovered residual consistency]\label{lem:recovered_residual_consistency}
Assume that $A\in L^\infty(\Omega)^{2\times2}$, $u\in W^{k+2,\infty}(\Omega)$, and the assumptions of Lemma~\ref{lem:hess_recovery_estimate} hold. Then
\begin{equation}\label{eq:recovered_residual_consistency}
    \|A:H_h\Pih u-f\|_{0,\Omega}
    \le
    Ch^k\|u\|_{W^{k+2,\infty}(\Omega)}.
\end{equation}
\end{lemma}

\begin{proof}
Since $f=A:D^2u$,
\[
    A:H_h\Pih u-f
    =
    A:(H_h\Pih u-D^2u).
\]
By the pointwise contraction bound for matrices,
\[
    |A:(H_h\Pih u-D^2u)|
    \le
    |A|\,|H_h\Pih u-D^2u|.
\]
Therefore,
\[
    \|A:H_h\Pih u-f\|_{0,\Omega}
    \le
    \|A\|_{L^\infty(\Omega)}
    \|H_h\Pih u-D^2u\|_{0,\Omega}.
\]
The desired estimate follows from the $L^2$ recovery estimate in Lemma~\ref{lem:hess_recovery_estimate}.
\end{proof}

\begin{proposition}[Consistency of the recovered-residual formulations]\label{prop:formulation_consistency}
Under the assumptions of Lemma~\ref{lem:recovered_residual_consistency}, the following consistency estimates hold.

\begin{enumerate}[label=\emph{(\roman*)},leftmargin=2.5em]
\item \emph{Scheme~1.} Assume that the coefficient values used in the nodal scheme are uniformly bounded at the interior nodes. Then
\begin{equation}\label{eq:scheme1_nodal_consistency}
    \max_{z_i\in\mathcal N_h^I}
    \left|
    A(z_i):H_h\Pih u(z_i)-f(z_i)
    \right|
    \le
    Ch^k\|u\|_{W^{k+2,\infty}(\Omega)}.
\end{equation}

\item \emph{Scheme~2.} For all $v_h\in V_{h,0}^k$,
\begin{equation}\label{eq:scheme2_consistency}
    \left|
    (A:H_h\Pih u,v_h)-(f,v_h)
    \right|
    \le
    Ch^k\|u\|_{W^{k+2,\infty}(\Omega)}
    \|v_h\|_{0,\Omega}.
\end{equation}

\item \emph{Scheme~3.} Since Scheme~3 is restricted to linear elements, let $\Pih u=I_h^1u$, and define
\begin{equation}\label{eq:delta_int_def}
    \delta_h^{\rm int}(u)
    :=
    \|I_h^1(A:H_h\Pih u)-A:H_h\Pih u\|_{0,\Omega}.
\end{equation}
Then, for all $w_h\in W_{h,0}^1$,
\begin{equation}\label{eq:scheme3_consistency}
    \left|
    (I_h^1(A:H_h\Pih u),w_h)-(f,w_h)
    \right|
    \le
    \left(
    Ch\|u\|_{W^{3,\infty}(\Omega)}
    +
    \delta_h^{\rm int}(u)
    \right)
    \|w_h\|_{0,\Omega}.
\end{equation}
\end{enumerate}
\end{proposition}

\begin{proof}
For Scheme~1, since $f(z_i)=A(z_i):D^2u(z_i)$,
\[
    A(z_i):H_h\Pih u(z_i)-f(z_i)
    =
    A(z_i):(H_h\Pih u(z_i)-D^2u(z_i)).
\]
The uniform boundedness of the nodal coefficient values gives
\[
    \left|
    A(z_i):H_h\Pih u(z_i)-f(z_i)
    \right|
    \le
    C\left|H_h\Pih u(z_i)-D^2u(z_i)\right|.
\]
Taking the maximum over $z_i\in\mathcal N_h^I$ and applying the nodal recovery estimate in Lemma~\ref{lem:hess_recovery_estimate} gives \eqref{eq:scheme1_nodal_consistency}.

For Scheme~2, by the definition of $r_h(u)$,
\[
    (A:H_h\Pih u,v_h)-(f,v_h)=(r_h(u),v_h).
\]
The estimate \eqref{eq:scheme2_consistency} follows from Lemma~\ref{lem:recovered_residual_consistency} and the Cauchy--Schwarz inequality.

For Scheme~3, we decompose the residual as
\[
\begin{aligned}
    (I_h^1(A:H_h\Pih u),w_h)-(f,w_h)
    &=(A:H_h\Pih u-f,w_h)  \\
    &\quad +(I_h^1(A:H_h\Pih u)-A:H_h\Pih u,w_h)  \\
    &=(r_h(u),w_h)
    +(I_h^1(A:H_h\Pih u)-A:H_h\Pih u,w_h).
\end{aligned}
\]
Hence, by the Cauchy--Schwarz inequality,
\[
    \left|
    (I_h^1(A:H_h\Pih u),w_h)-(f,w_h)
    \right|
    \le
    \big(
    \|r_h(u)\|_{0,\Omega}+
    \delta_h^{\rm int}(u)
    \big)
    \|w_h\|_{0,\Omega}.
\]
Using Lemma~\ref{lem:recovered_residual_consistency} with $k=1$ gives \eqref{eq:scheme3_consistency}. This completes the proof.
\end{proof}

Together, Lemma~\ref{lem:recovered_residual_consistency} and Proposition~\ref{prop:formulation_consistency} show that the recovered residual is consistent in the natural form associated with each discretization.

\begin{remark}\label{rem:scheme3_delta_int}
The term $\delta_h^{\rm int}(u)$ in the consistency estimate for Scheme~3 measures the interpolation error introduced when the recovered residual is first interpolated and then tested with the biorthogonal basis. If $A:H_h\Pi_hu$ has sufficient elementwise regularity, this term can be bounded by standard interpolation estimates. In the nonsmooth or discontinuous coefficient settings considered in Section~\ref{sec:numerics}, however, such a bound may depend on the chosen nodal representative and on the regularity of the recovered residual across coefficient interfaces. We therefore keep $\delta_h^{\rm int}(u)$ explicit.
\end{remark}

\begin{corollary}[Nodal error estimate in the globally monotone regime]
\label{cor:nodal_error_monotone}
Assume that the hypotheses of Proposition~\ref{prop:formulation_consistency}\emph{(i)} hold for Scheme~1 and that the recovered nodal matrix satisfies the global monotonicity condition of Definition~\ref{def:global_monotone}. Let $u_h\in V_h^k$ be the solution of Scheme~1 with exact Dirichlet nodal values, and let $\Pi_hu=I_h^ku$ be the Lagrange interpolant of the exact solution. Then
\begin{equation}\label{eq:nodal_error_corollary}
    \|(\Pi_hu-u_h)_I\|_{\ell^\infty}
    \le
    Ch^k\|u\|_{W^{k+2,\infty}(\Omega)}.
\end{equation}
Consequently, for fixed polynomial degree $k$,
\begin{equation}\label{eq:function_error_interp}
    \|\Pi_hu-u_h\|_{L^\infty(\Omega)}
    \le
    Ch^k\|u\|_{W^{k+2,\infty}(\Omega)}.
\end{equation}
Combining this estimate with the standard interpolation error gives
\begin{equation}\label{eq:function_error_exact}
    \|u-u_h\|_{L^\infty(\Omega)}
    \le
    Ch^k\|u\|_{W^{k+2,\infty}(\Omega)}.
\end{equation}
\end{corollary}

\begin{proof}
Let $U_h$ denote the nodal vector of $u_h$, and let $U^*$ denote the nodal vector of $\Pi_hu$. Since $u_h$ is imposed with exact Dirichlet nodal values and $\Pi_hu$ interpolates $u$, we have
\[
    (U^*-U_h)_B=0.
\]
Set $e_I=(U^*-U_h)_I$. By the matrix form of Scheme~1,
\[
    K_{1,h}(U_h)_I
    =
    F_{1,I}-(K_A)_{IB}(U_h)_B.
\]
On the other hand, by the definition of the full recovered nodal matrix $K_A$,
\[
    (K_AU^*)_I
    =
    K_{1,h}(U^*)_I+(K_A)_{IB}(U^*)_B.
\]
Since $(U^*)_B=(U_h)_B$, subtracting the discrete equation from the last identity gives
\[
    K_{1,h}e_I=(K_AU^*)_I-F_{1,I}.
\]
For each interior node $z_i$, the $i$th component of the right-hand side is
\[
    \bigl((K_AU^*)_I-F_{1,I}\bigr)_i
    =
    A(z_i):H_h\Pi_hu(z_i)-f(z_i).
\]
Define the interior nodal residual vector $\tau_I$ by
\[
    (\tau_I)_i=A(z_i):H_h\Pi_hu(z_i)-f(z_i),
    \qquad z_i\in\mathcal N_h^I.
\]
Then $K_{1,h}e_I=\tau_I$. Since $L_h=-K_{1,h}$,
\[
    L_he_I=-\tau_I.
\]
By Theorem~\ref{thm:cond_monotone}, the global monotonicity assumption gives the uniform inverse bound $\|L_h^{-1}\|_\infty\le C$. Therefore,
\[
    \|e_I\|_{\ell^\infty}
    \le
    \|L_h^{-1}\|_\infty\|\tau_I\|_{\ell^\infty}
    \le
    C\|\tau_I\|_{\ell^\infty}.
\]
Using the nodal residual consistency estimate in Proposition~\ref{prop:formulation_consistency}\emph{(i)}, we obtain
\[
    \|\tau_I\|_{\ell^\infty}
    \le
    Ch^k\|u\|_{W^{k+2,\infty}(\Omega)}.
\]
This proves \eqref{eq:nodal_error_corollary}.

The function $\Pi_hu-u_h$ belongs to $V_h^k$, has zero boundary nodal values, and has interior nodal vector $e_I$. For fixed polynomial degree $k$, the $L^\infty$ stability of the Lagrange nodal basis on a shape-regular mesh family gives
\[
    \|\Pi_hu-u_h\|_{L^\infty(\Omega)}
    \le C\|e_I\|_{\ell^\infty},
\]
which yields \eqref{eq:function_error_interp}. Finally, the standard interpolation estimate gives
\[
    \|u-\Pi_hu\|_{L^\infty(\Omega)}
    \le
    Ch^{k+1}\|u\|_{W^{k+1,\infty}(\Omega)}.
\]
Since $h\le1$ and $u\in W^{k+2,\infty}(\Omega)$, the triangle inequality gives \eqref{eq:function_error_exact}.
\end{proof}

\begin{remark}
Corollary~\ref{cor:nodal_error_monotone} does not introduce a new solvability assumption. It is a consequence of the globally monotone algebraic regime: Theorem~\ref{thm:unique_global_monotone} gives unique solvability, Theorem~\ref{thm:cond_monotone} gives the uniform inverse bound, and Proposition~\ref{prop:formulation_consistency}\emph{(i)} gives nodal residual consistency. The result is a nodal $L^\infty$ consequence for Scheme~1 and is not intended as a full optimal $L^2$- or $H^1$-error theory. Scheme~2 and Scheme~3 are not covered by this corollary.
\end{remark}

\section{Numerical experiments}\label{sec:numerics}

The numerical study is designed to examine both the accuracy and the algebraic behavior of the Hessian-recovery-based nonvariational discretizations. We consider three increasingly demanding settings: continuous but nonsmooth coefficients, discontinuous coefficients, and a Monge--Amp\`ere type fully nonlinear equation. The first example tests the globally monotone and localized sign-violation regimes identified in Section~\ref{sec:solvability}. The second isolates the effect of coefficient jumps on unstructured meshes while keeping a smooth manufactured solution for rate verification. The third illustrates how the recovered Hessian can be used within a Newton linearization, where each linearized equation has the structure of a variable-coefficient non-divergence form problem.

The tables below are organized to connect the numerical behavior with the preceding analysis. We first report matrix diagnostics, which indicate whether the nodal operator falls into the globally monotone regime or into the localized sign-violation regime. We then report solution and recovered-Hessian errors to assess the accuracy of the corresponding recovered-residual discretizations.

\subsection{Error measures and algebraic diagnostics}\label{subsec:error_diag}

We use the following error measures in the numerical experiments:
\[
    e_0=\|u-u_h\|_{0,\Omega},
    \qquad
    e_1=|u-u_h|_{1,\Omega},
\]
and
\[
    e_H=\|D^2u-H_hu_h\|_{0,\Omega},
    \qquad
    e_\Delta=\|\Delta u-\operatorname{tr}(H_hu_h)\|_{0,\Omega}.
\]
For the $P_2$ tests, we also report the broken $H^2$-type error
\[
    e_2=|u-u_h|_{2,h}
    :=\left(\sum_{T\in\mathcal T_h}\|D^2u-D^2u_h\|_{0,T}^2\right)^{1/2}.
\]
The convergence rate for an error $e$ is computed as
\[
    \operatorname{rate}
    =\log_2\frac{e(h)}{e(h/2)}.
\]

In addition to the error tables, we report algebraic diagnostics for the sign-reversed nodal matrix $L_h$. Unless otherwise stated, the sign-violation diagnostics are computed over the full interior rows, namely for $i\in I$ and $j\in I\cup B$ with $j\ne i$, with the diagonal entries excluded. We denote by $N_{\rm bad}$ the number of interior rows containing at least one positive off-diagonal entry, and by $\rho_{\rm bad}=N_{\rm bad}/n_I$ the corresponding bad-row ratio. We also denote by $N_+$ the total number of positive off-diagonal entries and by $\rho_+$ their ratio among the nonzero off-diagonal entries in these full interior rows. On uniform meshes, we check whether $L_h$ falls into the globally monotone regime of Section~\ref{sec:solvability} by recording the row-sum error, $N_+$, and $N_{\rm bad}$. On unstructured meshes, we further report the Schur-complement diagnostics $\sigma_{\min}(\mathfrak S_h)$ and $\kappa_2(\mathfrak S_h)$, where $\sigma_{\min}$ denotes the smallest singular value and $\kappa_2$ the spectral condition number. These diagnostics are used to examine the algebraic criteria in Section~\ref{sec:solvability}.

The same PPR Hessian recovery construction is used for all schemes on a given mesh. Near the boundary, recovery patches are enlarged when necessary to maintain local unisolvence and stable scaled least-squares fitting. In the Schur-complement diagnostics, the exterior connectivity condition for the good block is checked before forming $\mathfrak S_h$. These Schur-complement quantities are therefore not used merely as independent performance indicators; they are reported as the assembled-matrix counterpart of the localized solvability criterion in Theorem~\ref{thm:schur_unique}. Implementation details and codes will be made available in an accompanying repository.

\subsection{Example 1: nonsmooth coefficients}\label{subsec:ex1_nonsmooth}

We first consider a continuous but nonsmooth coefficient matrix. This example is used to examine both algebraic regimes discussed in Section~\ref{sec:solvability}: the globally monotone regime on uniform meshes and the localized sign-violation regime on unstructured meshes. It also provides the main accuracy test for the three recovered-residual formulations.

We solve \eqref{eq:model} on $\Omega=(-1,1)^2$ with the exact solution
\[
    u(x,y)=\sin x\sin y.
\]
The coefficient matrix is
\begin{equation}\label{eq:ex1_coeff}
    A(x,y)=
    \begin{pmatrix}
        1+|x| & \frac12 |xy|^{1/3}\\[1mm]
        \frac12 |xy|^{1/3} & 1+|y|
    \end{pmatrix}.
\end{equation}
The right-hand side and boundary data are prescribed by
\[
    f=A:D^2u,
    \qquad
    g=u|_{\partial\Omega}.
\]
The coefficient matrix is continuous and uniformly positive definite, but it is not smooth. Thus the example tests the recovered operator in a setting where the coefficient regularity is limited, while the exact solution remains smooth enough to provide clear convergence rates.

Table~\ref{tab:strict_monotone_uniform_ex1} shows that the uniform meshes fall into the clean algebraic regime described in Section~\ref{sec:solvability}. The row-sum defects remain at round-off level, and no positive off-diagonal entries or bad rows are detected. Hence the sign-reversed nodal matrices satisfy the global monotonicity condition. The condition-number estimates increase by approximately a factor of four under uniform refinement, in agreement with the $O(h^{-2})$ bound in Theorem~\ref{thm:cond_monotone}.

\begin{table}[H]
\centering
\footnotesize
\setlength{\tabcolsep}{3pt}
\caption{Uniform-mesh monotonicity diagnostics for Example~1.}
\label{tab:strict_monotone_uniform_ex1}
\begin{tabular}{ccccccc}
\toprule
Scheme & $n_I$ & $N_{\rm bad}$ & $N_+$ & row-sum error
& $\operatorname{condest}$ & strict sign \\
\midrule
1 & 961   & 0 & 0 & $4.6612{\times}10^{-12}$ & $1.2756{\times}10^3$ & yes \\
3 & 961   & 0 & 0 & $1.1657{\times}10^{-14}$ & $1.2756{\times}10^3$ & yes \\
1 & 3969  & 0 & 0 & $1.8645{\times}10^{-11}$ & $5.2290{\times}10^3$ & yes \\
3 & 3969  & 0 & 0 & $1.1713{\times}10^{-14}$ & $5.2290{\times}10^3$ & yes \\
1 & 16129 & 0 & 0 & $7.4579{\times}10^{-11}$ & $2.1108{\times}10^4$ & yes \\
3 & 16129 & 0 & 0 & $1.1990{\times}10^{-14}$ & $2.1108{\times}10^4$ & yes \\
\bottomrule
\end{tabular}
\end{table}

On unstructured meshes, the global off-diagonal sign condition is no longer preserved. Table~\ref{tab:ex1_three_scheme_matrix_diagnostics} shows that Scheme~1 and Scheme~3 develop sign-violating rows, but these violations remain localized and decrease in relative frequency under refinement. After separating the bad rows, the associated Schur complements remain nonsingular and moderately conditioned in all reported cases. The perturbation matrix $E_h=L_{SS}^{-1}L_{SG}L_{GG}^{-1}L_{GS}$ from Remark~\ref{rem:schur_criterion} was also examined; in the reported Scheme~1 and Scheme~3 cases, its spectral radius remained below one, while fixed norm bounds were more conservative. The table therefore reports the direct Schur-complement diagnostics. This behavior is consistent with the localized algebraic solvability mechanism of Theorem~\ref{thm:schur_unique}.

Scheme~2 exhibits a different algebraic pattern. Because the Galerkin-type formulation contains mass-type averaging, many more rows contain positive off-diagonal entries, and the nodal Schur-complement mechanism is not applied to this matrix. At the same time, Scheme~2 has smaller condition-number estimates in this test. Full singular-value checks on the first three unstructured meshes gave positive smallest singular values for the Scheme~2 matrices. These observations suggest favorable computational behavior for the Galerkin-type formulation in the tests considered here, although its algebraic behavior is not explained by the nodal matrix mechanism analyzed in Section~\ref{sec:solvability}.

\begin{table}[H]
\centering
\footnotesize
\setlength{\tabcolsep}{3pt}
\caption{Unstructured-mesh diagnostics for Example~1.}
\label{tab:ex1_three_scheme_matrix_diagnostics}
\begin{tabular}{cccccccccc}
\toprule
$n_I$ & Scheme
& $\operatorname{condest}$
& $N_{\rm bad}$
& $\rho_{\rm bad}$
& $N_+$
& $\rho_+$
& $\sigma_{\min}(\mathfrak S_h)$
& $\kappa_2(\mathfrak S_h)$
& row sum \\
\midrule
433 & 1 & $7.4255{\times}10^2$ & 64 & $14.78\%$ & 81 & $3.12\%$ & $4.0264{\times}10^1$ & $3.4887{\times}10^1$ & $2.1600{\times}10^{-12}$ \\
433 & 2 & $4.7337{\times}10^2$ & 326 & $75.29\%$ & 793 & $11.30\%$ & -- & -- & $1.1546{\times}10^{-14}$ \\
433 & 3 & $6.7028{\times}10^2$ & 64 & $14.78\%$ & 81 & $3.12\%$ & $2.9330{\times}10^{-1}$ & $2.9407{\times}10^1$ & $1.4877{\times}10^{-14}$ \\
\midrule
1809 & 1 & $3.2768{\times}10^3$ & 185 & $10.23\%$ & 235 & $2.17\%$ & $5.9974{\times}10^1$ & $1.0632{\times}10^2$ & $9.3223{\times}10^{-12}$ \\
1809 & 2 & $1.6053{\times}10^3$ & 1144 & $63.24\%$ & 2528 & $7.80\%$ & -- & -- & $1.1838{\times}10^{-14}$ \\
1809 & 3 & $2.8350{\times}10^3$ & 185 & $10.23\%$ & 235 & $2.17\%$ & $1.1240{\times}10^{-1}$ & $8.6779{\times}10^1$ & $1.5987{\times}10^{-14}$ \\
\midrule
7393 & 1 & $1.3061{\times}10^4$ & 428 & $5.79\%$ & 544 & $1.23\%$ & $1.0495{\times}10^2$ & $2.5950{\times}10^2$ & $4.5475{\times}10^{-11}$ \\
7393 & 2 & $6.2784{\times}10^3$ & 4031 & $54.52\%$ & 8399 & $6.33\%$ & -- & -- & $1.2768{\times}10^{-14}$ \\
7393 & 3 & $1.1421{\times}10^4$ & 428 & $5.79\%$ & 544 & $1.23\%$ & $4.9800{\times}10^{-2}$ & $2.0973{\times}10^2$ & $1.6653{\times}10^{-14}$ \\
\midrule
29889 & 1 & $5.3342{\times}10^4$ & 916 & $3.06\%$ & 1161 & $0.65\%$ & $1.9710{\times}10^2$ & $5.6480{\times}10^2$ & $2.1646{\times}10^{-10}$ \\
29889 & 2 & $2.4842{\times}10^4$ & 14417 & $48.24\%$ & 29775 & $5.54\%$ & -- & -- & $1.3545{\times}10^{-14}$ \\
29889 & 3 & $4.6012{\times}10^4$ & 916 & $3.06\%$ & 1161 & $0.65\%$ & $2.3500{\times}10^{-2}$ & $4.5418{\times}10^2$ & $2.0206{\times}10^{-14}$ \\
\bottomrule
\end{tabular}
\end{table}

To further examine the transition from the globally monotone regime to the localized sign-violation regime, we include a controlled mesh-perturbation test. Starting from the uniform mesh used in Example~1, each interior vertex $z_i$ is moved to
\[
    z_i^\delta=z_i+\delta h\xi_i,
\]
where $\delta\ge0$ is a prescribed nondimensional perturbation amplitude, $h$ is the uniform mesh size, $\xi_i\in[-1,1]^2$ is generated with a fixed random seed, and all boundary vertices are kept fixed. Perturbations that produce inverted elements are rejected.

Table~\ref{tab:mesh_perturb_ex1} reports the Scheme~1 diagnostics. Since this auxiliary test is designed to track the loss of the nodal sign structure under mesh perturbations, we report only Scheme~1. Scheme~3 is not listed separately because $K_{3,h}=D_\mu K_{1,h}$, so it has the same sign-violation pattern and inherits the same nonsingularity conclusion through positive diagonal row scaling. The unperturbed case $\delta=0$ serves as the monotone baseline. As $\delta$ increases, positive off-diagonal entries appear gradually, indicating that the global monotone sign structure is sensitive to mesh perturbations. Nevertheless, for the perturbation levels considered here, the sign-violating rows remain localized, the row-sum defects stay at roundoff level, and the associated Schur complements remain nonsingular with moderate condition numbers. This controlled test illustrates how the Schur-complement diagnostic provides an assembled-matrix solvability certificate after a mild loss of global monotonicity.

\begin{table}[H]
\centering
\footnotesize
\setlength{\tabcolsep}{3pt}
\caption{Controlled mesh-perturbation diagnostics for Example~1. Here $n_I=3969$ for all perturbation levels.}
\label{tab:mesh_perturb_ex1}
\begin{tabular}{cccccc}
\toprule
$\delta$ & $\operatorname{condest}$ & $N_{\rm bad}\, (\rho_{\rm bad})$ & $N_+\, (\rho_+)$ & $\sigma_{\min}(\mathfrak S_h)$ & $\kappa_2(\mathfrak S_h)$ \\
\midrule
0.00 & $5.2290{\times}10^3$ & 0 (0.00\%) & 0 (0.00\%) & -- & -- \\
0.01 & $5.2157{\times}10^3$ & 65 (1.64\%) & 65 (0.27\%) & $1.0850{\times}10^2$ & $7.3515{\times}10^1$ \\
0.02 & $5.2947{\times}10^3$ & 103 (2.60\%) & 104 (0.44\%) & $6.5546{\times}10^1$ & $1.2740{\times}10^2$ \\
0.03 & $5.5008{\times}10^3$ & 218 (5.49\%) & 219 (0.92\%) & $3.6363{\times}10^1$ & $2.5418{\times}10^2$ \\
0.04 & $5.7061{\times}10^3$ & 382 (9.62\%) & 383 (1.61\%) & $2.5170{\times}10^1$ & $3.6863{\times}10^2$ \\
0.05 & $5.9102{\times}10^3$ & 648 (16.33\%) & 649 (2.73\%) & $1.8122{\times}10^1$ & $5.1839{\times}10^2$ \\
\bottomrule
\end{tabular}
\end{table}

Table~\ref{tab:ex1_p1_errors_three_schemes} reports the $P_1$ errors on unstructured meshes. All three schemes achieve the expected rates for the finite element solution, with approximately second-order convergence in $e_0$ and first-order convergence in $e_1$. Scheme~1 and Scheme~2 give nearly identical solution errors. Scheme~3 preserves the same $L^2$- and $H^1$-error rates, although its recovered Hessian and recovered Laplacian errors are larger in this test. For Scheme~1 and Scheme~2, the recovered Laplacian error $e_\Delta$ is close to second order, while the recovered Hessian error $e_H$ converges with an order around $1.5$.

\begin{table}[H]
\centering
\footnotesize
\setlength{\tabcolsep}{3pt}
\caption{$P_1$ errors on unstructured meshes for Example~1.}
\label{tab:ex1_p1_errors_three_schemes}
\begin{tabular}{cccccccccc}
\toprule
\multirow{2}{*}{Scheme} & \multirow{2}{*}{$n_I$}
& \multicolumn{2}{c}{$e_0$}
& \multicolumn{2}{c}{$e_1$}
& \multicolumn{2}{c}{$e_H$}
& \multicolumn{2}{c}{$e_\Delta$} \\
\cmidrule(lr){3-4}\cmidrule(lr){5-6}\cmidrule(lr){7-8}\cmidrule(lr){9-10}
& & error & order & error & order & error & order & error & order \\
\midrule
1 & 105 & $3.80{\times}10^{-3}$ & -- & $1.23{\times}10^{-1}$ & -- & $8.28{\times}10^{-2}$ & -- & $6.47{\times}10^{-2}$ & -- \\
1 & 433 & $9.59{\times}10^{-4}$ & 1.99 & $6.15{\times}10^{-2}$ & 1.00 & $2.70{\times}10^{-2}$ & 1.62 & $1.44{\times}10^{-2}$ & 2.17 \\
1 & 1809 & $2.40{\times}10^{-4}$ & 2.00 & $3.08{\times}10^{-2}$ & 1.00 & $9.30{\times}10^{-3}$ & 1.54 & $3.60{\times}10^{-3}$ & 2.00 \\
1 & 7393 & $6.00{\times}10^{-5}$ & 2.00 & $1.54{\times}10^{-2}$ & 1.00 & $3.30{\times}10^{-3}$ & 1.49 & $9.30{\times}10^{-4}$ & 1.95 \\
\midrule
2 & 105 & $3.70{\times}10^{-3}$ & -- & $1.23{\times}10^{-1}$ & -- & $7.74{\times}10^{-2}$ & -- & $5.06{\times}10^{-2}$ & -- \\
2 & 433 & $9.42{\times}10^{-4}$ & 1.97 & $6.15{\times}10^{-2}$ & 1.00 & $2.61{\times}10^{-2}$ & 1.57 & $1.10{\times}10^{-2}$ & 2.20 \\
2 & 1809 & $2.39{\times}10^{-4}$ & 1.98 & $3.08{\times}10^{-2}$ & 1.00 & $9.20{\times}10^{-3}$ & 1.50 & $2.90{\times}10^{-3}$ & 1.92 \\
2 & 7393 & $5.99{\times}10^{-5}$ & 1.99 & $1.54{\times}10^{-2}$ & 1.00 & $3.30{\times}10^{-3}$ & 1.48 & $8.02{\times}10^{-4}$ & 1.85 \\
\midrule
3 & 105 & $4.70{\times}10^{-3}$ & -- & $1.23{\times}10^{-1}$ & -- & $8.57{\times}10^{-2}$ & -- & $6.91{\times}10^{-2}$ & -- \\
3 & 433 & $1.30{\times}10^{-3}$ & 1.85 & $6.15{\times}10^{-2}$ & 1.00 & $2.89{\times}10^{-2}$ & 1.57 & $1.84{\times}10^{-2}$ & 1.91 \\
3 & 1809 & $2.83{\times}10^{-4}$ & 2.20 & $3.08{\times}10^{-2}$ & 1.00 & $1.09{\times}10^{-2}$ & 1.41 & $7.10{\times}10^{-3}$ & 1.37 \\
3 & 7393 & $6.43{\times}10^{-5}$ & 2.14 & $1.54{\times}10^{-2}$ & 1.00 & $4.50{\times}10^{-3}$ & 1.28 & $3.50{\times}10^{-3}$ & 1.02 \\
\bottomrule
\end{tabular}
\end{table}

Table~\ref{tab:ex1_p2_errors} reports the $P_2$ results for Scheme~1 and Scheme~2. The results show that the recovered-residual construction extends naturally to higher-order Lagrange spaces. Both schemes achieve the expected finite element convergence rates: approximately third order in $e_0$, second order in $e_1$, and first order in the broken $H^2$-type error $e_2$. The recovered Hessian and Laplacian errors also decrease under refinement. Scheme~2 gives smaller recovery-related errors on several meshes, while the solution errors of the two schemes become essentially indistinguishable on finer meshes.

\begin{table}[H]
\centering
\footnotesize
\setlength{\tabcolsep}{3pt}
\caption{$P_2$ errors on unstructured meshes for Example~1.}
\label{tab:ex1_p2_errors}
\begin{tabular}{cccccccccccc}
\toprule
\multirow{2}{*}{Scheme} & \multirow{2}{*}{$n_I$}
& \multicolumn{2}{c}{$e_0$}
& \multicolumn{2}{c}{$e_1$}
& \multicolumn{2}{c}{$e_2$}
& \multicolumn{2}{c}{$e_H$}
& \multicolumn{2}{c}{$e_\Delta$} \\
\cmidrule(lr){3-4}\cmidrule(lr){5-6}\cmidrule(lr){7-8}\cmidrule(lr){9-10}\cmidrule(lr){11-12}
& & error & order & error & order & error & order & error & order & error & order \\
\midrule
1 & 433 & $8.50{\times}10^{-5}$ & -- & $3.29{\times}10^{-3}$ & -- & $1.25{\times}10^{-1}$ & -- & $7.26{\times}10^{-3}$ & -- & $4.29{\times}10^{-3}$ & -- \\
1 & 1809 & $1.09{\times}10^{-5}$ & 2.96 & $7.99{\times}10^{-4}$ & 2.04 & $6.20{\times}10^{-2}$ & 1.01 & $3.01{\times}10^{-3}$ & 1.27 & $1.42{\times}10^{-3}$ & 1.60 \\
1 & 7393 & $1.08{\times}10^{-6}$ & 3.34 & $1.72{\times}10^{-4}$ & 2.21 & $2.97{\times}10^{-2}$ & 1.06 & $5.99{\times}10^{-5}$ & 5.65 & $2.72{\times}10^{-5}$ & 5.70 \\
1 & 29889 & $1.35{\times}10^{-7}$ & 3.00 & $4.31{\times}10^{-5}$ & 2.00 & $1.48{\times}10^{-2}$ & 1.00 & $2.34{\times}10^{-5}$ & 1.35 & $8.51{\times}10^{-6}$ & 1.68 \\
1 & 120193 & $1.69{\times}10^{-8}$ & 3.00 & $1.08{\times}10^{-5}$ & 2.00 & $7.42{\times}10^{-3}$ & 1.00 & $2.81{\times}10^{-7}$ & 6.38 & $9.38{\times}10^{-8}$ & 6.50 \\
\midrule
2 & 433 & $7.23{\times}10^{-5}$ & -- & $2.83{\times}10^{-3}$ & -- & $1.19{\times}10^{-1}$ & -- & $2.70{\times}10^{-3}$ & -- & $1.29{\times}10^{-3}$ & -- \\
2 & 1809 & $8.65{\times}10^{-6}$ & 3.06 & $6.89{\times}10^{-4}$ & 2.04 & $5.93{\times}10^{-2}$ & 1.01 & $1.28{\times}10^{-4}$ & 4.40 & $4.68{\times}10^{-5}$ & 4.78 \\
2 & 7393 & $1.08{\times}10^{-6}$ & 3.00 & $1.72{\times}10^{-4}$ & 2.00 & $2.97{\times}10^{-2}$ & 1.00 & $2.37{\times}10^{-5}$ & 2.43 & $9.35{\times}10^{-6}$ & 2.32 \\
2 & 29889 & $1.35{\times}10^{-7}$ & 3.00 & $4.31{\times}10^{-5}$ & 2.00 & $1.48{\times}10^{-2}$ & 1.00 & $6.10{\times}10^{-6}$ & 1.96 & $1.89{\times}10^{-6}$ & 2.31 \\
2 & 120193 & $1.69{\times}10^{-8}$ & 3.00 & $1.08{\times}10^{-5}$ & 2.00 & $7.42{\times}10^{-3}$ & 1.00 & $1.12{\times}10^{-6}$ & 2.44 & $3.12{\times}10^{-7}$ & 2.60 \\
\bottomrule
\end{tabular}
\end{table}

The comparison in Example~1 highlights the separation between matrix structure and computational behavior. Scheme~1 is the formulation directly tied to the nodal algebraic criteria of Section~\ref{sec:solvability} and gives competitive accuracy. Scheme~3 inherits the same nonsingularity mechanism through the positive diagonal row scaling, but its recovered Hessian and Laplacian errors are less favorable in this test. Scheme~2 has a different matrix structure because of Galerkin-type averaging; nevertheless, it gives the smallest condition-number estimates and solution errors comparable to those of Scheme~1. This makes Scheme~2 a useful computational counterpart to the nodal formulation, while also illustrating why its analysis requires tools different from the nodal sign-structure arguments.

\subsection{Example 2: discontinuous coefficients}\label{subsec:ex2_disc}

We next consider a discontinuous coefficient matrix. The purpose of this manufactured-solution test is to isolate the effect of coefficient jumps on the recovered non-divergence operator and on the associated algebraic diagnostics, while keeping the exact solution smooth enough for a clear convergence-rate study.

We take
\[
    A(x,y)=
    \begin{pmatrix}
        2 & s(x,y)\\
        s(x,y) & 2
    \end{pmatrix},
    \qquad
    s(x,y)=\operatorname{sgn}(xy),
\]
on $\Omega=(-1,1)^2$. The convention $s(x,y)=0$ on the coordinate axes is used to fix the nodal representative on the coefficient jump set. The exact solution is
\[
    u(x,y)=\sin x\sin y.
\]
The right-hand side and boundary data are prescribed by
\[
    f=A:D^2u,
    \qquad
    g=u|_{\partial\Omega}.
\]

The coefficient jump leads to a more challenging algebraic pattern on unstructured meshes. As shown in Table~\ref{tab:ex2_p1_matrix_diagnostics}, Scheme~1 no longer falls into the globally monotone regime, and the bad-row ratio is substantially larger than in Example~1. Nevertheless, the sign-violating rows remain localized in the sense that their relative frequency decreases under refinement, and the associated Schur complements remain nonsingular in all reported cases. The spectral radius of the perturbation matrix in Remark~\ref{rem:schur_criterion} remained below one in the computed cases, although norm-based bounds were conservative. This supports the use of the localized Schur-complement diagnostic for the nodal formulation in the tested cases.

Scheme~2 again displays a different algebraic structure. The mass-type averaging produces positive off-diagonal entries in almost all rows, so the nodal Schur-complement diagnostic is not applied to this matrix. Its condition-number estimates, however, remain smaller than those of Scheme~1 in this test. Full singular-value checks on the first three meshes gave positive smallest singular values for the Scheme~2 matrices, consistent with the computational behavior observed in the tests of Example~1.

\begin{table}[H]
\centering
\footnotesize
\setlength{\tabcolsep}{3pt}
\caption{Unstructured-mesh diagnostics for Example~2.}
\label{tab:ex2_p1_matrix_diagnostics}
\begin{tabular}{cccccccc}
\toprule
$n_I$ & Scheme
& $\operatorname{condest}$
& $N_{\rm bad}$
& $\rho_{\rm bad}$
& $\rho_+$
& $\operatorname{condest}(\mathfrak S_h)$
& row sum \\
\midrule
433 & 1 & $1.3597{\times}10^{3}$ & 177 & $40.88\%$ & $8.05\%$ & $2.7408{\times}10^{2}$ & $2.7498{\times}10^{-12}$ \\
433 & 2 & $7.6234{\times}10^{2}$ & 422 & $97.46\%$ & $14.92\%$ & -- & $1.4211{\times}10^{-14}$ \\
\midrule
1809 & 1 & $8.6528{\times}10^{3}$ & 507 & $28.03\%$ & $6.09\%$ & $9.5773{\times}10^{2}$ & $1.2392{\times}10^{-11}$ \\
1809 & 2 & $4.3941{\times}10^{3}$ & 1786 & $98.73\%$ & $13.88\%$ & -- & $1.4710{\times}10^{-14}$ \\
\midrule
7393 & 1 & $4.8994{\times}10^{4}$ & 1431 & $19.36\%$ & $4.72\%$ & $2.7248{\times}10^{3}$ & $5.3433{\times}10^{-11}$ \\
7393 & 2 & $2.0510{\times}10^{4}$ & 7346 & $99.36\%$ & $13.08\%$ & -- & $1.7042{\times}10^{-14}$ \\
\midrule
29889 & 1 & $2.8405{\times}10^{5}$ & 4367 & $14.61\%$ & $3.97\%$ & $8.4087{\times}10^{3}$ & $2.1282{\times}10^{-10}$ \\
29889 & 2 & $1.1326{\times}10^{5}$ & 29779 & $99.63\%$ & $12.68\%$ & -- & $1.7431{\times}10^{-14}$ \\
\bottomrule
\end{tabular}
\end{table}

Table~\ref{tab:ex2_p2_errors} shows that the discontinuous-coefficient test retains the expected accuracy on the unstructured meshes considered here. Both Scheme~1 and Scheme~2 achieve the expected finite element convergence rates, with approximately third-order convergence in $e_0$, second-order convergence in $e_1$, and first-order convergence in the broken $H^2$-type error $e_2$. Scheme~2 gives smaller error constants for most quantities, while the two schemes exhibit comparable asymptotic rates. These results indicate that the recovered-residual discretizations can retain high-order accuracy in this manufactured discontinuous-coefficient setting, even though the algebraic sign structure is substantially more complicated than in Example~1.

\begin{table}[H]
\centering
\footnotesize
\setlength{\tabcolsep}{3pt}
\caption{$P_2$ errors on unstructured meshes for Example~2.}
\label{tab:ex2_p2_errors}
\begin{tabular}{cccccccccccc}
\toprule
\multirow{2}{*}{Scheme} & \multirow{2}{*}{$n_I$}
& \multicolumn{2}{c}{$e_0$}
& \multicolumn{2}{c}{$e_1$}
& \multicolumn{2}{c}{$e_2$}
& \multicolumn{2}{c}{$e_H$}
& \multicolumn{2}{c}{$e_\Delta$} \\
\cmidrule(lr){3-4}\cmidrule(lr){5-6}\cmidrule(lr){7-8}\cmidrule(lr){9-10}\cmidrule(lr){11-12}
& & error & order & error & order & error & order & error & order & error & order \\
\midrule
1 & 433 & $8.82{\times}10^{-5}$ & -- & $3.30{\times}10^{-3}$ & -- & $1.27{\times}10^{-1}$ & -- & $9.20{\times}10^{-3}$ & -- & $7.14{\times}10^{-3}$ & -- \\
1 & 1809 & $9.24{\times}10^{-6}$ & 3.25 & $7.01{\times}10^{-4}$ & 2.24 & $5.96{\times}10^{-2}$ & 1.09 & $1.09{\times}10^{-3}$ & 3.08 & $6.53{\times}10^{-4}$ & 3.45 \\
1 & 7393 & $1.23{\times}10^{-6}$ & 2.91 & $1.86{\times}10^{-4}$ & 1.92 & $3.03{\times}10^{-2}$ & 0.98 & $1.28{\times}10^{-3}$ & -0.24 & $7.65{\times}10^{-4}$ & -0.23 \\
1 & 29889 & $1.37{\times}10^{-7}$ & 3.16 & $4.33{\times}10^{-5}$ & 2.10 & $1.49{\times}10^{-2}$ & 1.03 & $1.58{\times}10^{-4}$ & 3.02 & $8.61{\times}10^{-5}$ & 3.15 \\
1 & 120193 & $1.69{\times}10^{-8}$ & 3.02 & $1.08{\times}10^{-5}$ & 2.01 & $7.42{\times}10^{-3}$ & 1.00 & $1.37{\times}10^{-5}$ & 3.53 & $7.34{\times}10^{-6}$ & 3.55 \\
\midrule
2 & 433 & $7.14{\times}10^{-5}$ & -- & $2.83{\times}10^{-3}$ & -- & $1.19{\times}10^{-1}$ & -- & $2.69{\times}10^{-3}$ & -- & $1.48{\times}10^{-3}$ & -- \\
2 & 1809 & $8.69{\times}10^{-6}$ & 3.04 & $6.92{\times}10^{-4}$ & 2.03 & $5.94{\times}10^{-2}$ & 1.01 & $4.43{\times}10^{-4}$ & 2.60 & $2.59{\times}10^{-4}$ & 2.51 \\
2 & 7393 & $1.08{\times}10^{-6}$ & 3.00 & $1.73{\times}10^{-4}$ & 2.00 & $2.97{\times}10^{-2}$ & 1.00 & $1.21{\times}10^{-4}$ & 1.87 & $6.67{\times}10^{-5}$ & 1.96 \\
2 & 29889 & $1.35{\times}10^{-7}$ & 3.00 & $4.31{\times}10^{-5}$ & 2.00 & $1.48{\times}10^{-2}$ & 1.00 & $6.66{\times}10^{-5}$ & 0.86 & $3.53{\times}10^{-5}$ & 0.92 \\
2 & 120193 & $1.69{\times}10^{-8}$ & 3.00 & $1.08{\times}10^{-5}$ & 2.00 & $7.42{\times}10^{-3}$ & 1.00 & $1.06{\times}10^{-5}$ & 2.65 & $5.40{\times}10^{-6}$ & 2.71 \\
\bottomrule
\end{tabular}
\end{table}

\paragraph{A supplementary low-regularity test.}
We further consider a supplementary test based on the discontinuous coefficient matrix in Example~2, but with a less regular manufactured solution,
\[
    u(x,y)=|x|^{2+\alpha}\sin(\pi y),\qquad \alpha=\frac14 .
\]
This test complements the preceding smooth-solution experiments by reducing the regularity of the exact solution while keeping the same discontinuous-coefficient setting. Since the Hessian of $u$ has only limited H\"older regularity across $x=0$, the optimal recovery behavior observed for smooth solutions is not expected.

Table~\ref{tab:ex2_lowreg_p1_errors} reports the $P_1$ errors for Scheme~1 and Scheme~2 on unstructured meshes. Both formulations show consistent convergence behavior. The solution errors retain the expected $P_1$ behavior, with approximately second-order convergence in $e_0$ and first-order convergence in $e_1$. The recovered Hessian and recovered Laplacian errors also decrease steadily, but with reduced rates compared with the smooth-solution tests, in agreement with the limited regularity of the exact solution and the regularity requirements of Hessian recovery. In this supplementary test, Scheme~1 gives slightly smaller solution errors, while Scheme~2 gives slightly smaller recovered-derivative errors; overall, the two formulations exhibit comparable convergence behavior.

\begin{table}[H]
\centering
\footnotesize
\setlength{\tabcolsep}{3pt}
\caption{$P_1$ errors for the supplementary low-regularity test on unstructured meshes.}
\label{tab:ex2_lowreg_p1_errors}
\begin{tabular}{cccccccccc}
\toprule
\multirow{2}{*}{Scheme} & \multirow{2}{*}{$n_I$}
& \multicolumn{2}{c}{$e_0$}
& \multicolumn{2}{c}{$e_1$}
& \multicolumn{2}{c}{$e_H$}
& \multicolumn{2}{c}{$e_\Delta$} \\
\cmidrule(lr){3-4}\cmidrule(lr){5-6}\cmidrule(lr){7-8}\cmidrule(lr){9-10}
& & error & order & error & order & error & order & error & order \\
\midrule
1 & 139  & $1.88{\times}10^{-2}$ & -- & $5.714{\times}10^{-1}$ & -- & $2.2118$ & -- & $1.0669$ & -- \\
1 & 513  & $4.50{\times}10^{-3}$ & 2.06 & $2.872{\times}10^{-1}$ & 0.99 & $8.560{\times}10^{-1}$ & 1.37 & $4.680{\times}10^{-1}$ & 1.19 \\
1 & 1969 & $1.10{\times}10^{-3}$ & 2.03 & $1.437{\times}10^{-1}$ & 1.00 & $3.070{\times}10^{-1}$ & 1.48 & $1.784{\times}10^{-1}$ & 1.39 \\
1 & 7713 & $2.859{\times}10^{-4}$ & 1.94 & $7.18{\times}10^{-2}$ & 1.00 & $1.102{\times}10^{-1}$ & 1.48 & $6.72{\times}10^{-2}$ & 1.41 \\
\midrule
2 & 139  & $2.26{\times}10^{-2}$ & -- & $5.757{\times}10^{-1}$ & -- & $2.1030$ & -- & $9.882{\times}10^{-1}$ & -- \\
2 & 513  & $5.30{\times}10^{-3}$ & 2.09 & $2.877{\times}10^{-1}$ & 1.00 & $7.981{\times}10^{-1}$ & 1.40 & $4.101{\times}10^{-1}$ & 1.27 \\
2 & 1969 & $1.30{\times}10^{-3}$ & 2.03 & $1.437{\times}10^{-1}$ & 1.00 & $2.828{\times}10^{-1}$ & 1.50 & $1.523{\times}10^{-1}$ & 1.43 \\
2 & 7713 & $3.104{\times}10^{-4}$ & 2.07 & $7.18{\times}10^{-2}$ & 1.00 & $1.009{\times}10^{-1}$ & 1.49 & $5.71{\times}10^{-2}$ & 1.42 \\
\bottomrule
\end{tabular}
\end{table}

\subsection{Example 3: a fully nonlinear Monge--Amp\`ere type problem}\label{subsec:ma}

We finally consider a Monge--Amp\`ere type problem to illustrate the use of the recovered Hessian within a nonlinear iteration. The key feature of this example is that the coefficient matrix in each Newton step is not prescribed in advance. Instead, it is generated from the current recovered Hessian and changes during the iteration. Thus the computation tests the recovered-residual discretization on a sequence of dynamically generated variable-coefficient non-divergence form problems.

Let
\[
    \det D^2u=f \qquad \text{in }\Omega,
    \qquad
    u=g \qquad \text{on }\partial\Omega .
\]
Given $u^n$, the Newton linearization reads
\[
    \operatorname{cof}(D^2u^n):D^2u^{n+1}
    = f+\det(D^2u^n)
    \qquad \text{in }\Omega,
\]
with $u^{n+1}=g$ on $\partial\Omega$. Here $\operatorname{cof}(D^2u^n)$ denotes the cofactor matrix of the Hessian. Hence each Newton step has the form
\[
    A^n:D^2u^{n+1}=F^n,
    \qquad
    A^n=\operatorname{cof}(D^2u^n),
    \qquad
    F^n=f+\det(D^2u^n),
\]
which is a variable-coefficient equation in non-divergence form.

The initial guess is chosen as the solution of the Poisson problem
\[
    \Delta u^0=2\sqrt f
    \qquad \text{in }\Omega,
    \qquad
    u^0=g
    \qquad \text{on }\partial\Omega .
\]

We discretize this Newton sequence by replacing the Hessian with the recovered Hessian.

\paragraph{Recovered-Hessian Newton iteration.}
In the discrete Newton iteration, both the coefficient matrix and the right-hand side are updated from the recovered Hessian. Let $u_h^0\in V_{h,g}^k$ be the finite element approximation of the initial guess. Given $u_h^n\in V_{h,g}^k$, find $u_h^{n+1}\in V_{h,g}^k$ such that
\begin{equation}\label{eq:ma_newton_discrete}
\operatorname{cof}\!\left(H_hu_h^n(z_i)\right)
:
H_hu_h^{n+1}(z_i)
=
 f(z_i)+\det\!\left(H_hu_h^n(z_i)\right),
\qquad z_i\in\mathcal N_h^I .
\end{equation}
The boundary nodal values of $u_h^{n+1}$ are prescribed by $g$. Equivalently, each nonlinear step solves a recovered non-divergence equation with coefficient matrix
\[
    A_h^n(z_i)=\operatorname{cof}\!\left(H_hu_h^n(z_i)\right)
\]
and right-hand side
\[
    F_h^n(z_i)=f(z_i)+\det\!\left(H_hu_h^n(z_i)\right).
\]
Thus the same recovered nodal operator used in the linear tests is applied here to a sequence of Newton coefficients generated dynamically by the nonlinear iteration.

We use the following manufactured solution. Let
\[
    \Omega=(-1,1)^2,
    \qquad
    u(x,y)=\frac12(x^2+y^2)+\varepsilon\sin(\pi x)\sin(\pi y),
    \qquad \varepsilon=0.02.
\]
The boundary data and right-hand side are prescribed by
\[
    g=u|_{\partial\Omega},
    \qquad
    f=\det D^2u.
\]
The parameter $\varepsilon$ is chosen so that the exact solution is strictly convex while still containing a nontrivial mixed derivative.

Table~\ref{tab:ma_nonuniform_diagnostics} reports the worst algebraic diagnostics over the entire Newton history on each mesh. Since
\[
    A_h^n=\operatorname{cof}(H_hu_h^n)
\]
is updated at every Newton step, the sign pattern, the bad-row set, the Schur complement, and the conditioning of the nodal matrix are not fixed a priori. The diagnostics therefore examine whether the localized Schur-complement mechanism remains informative for these dynamically generated linearized problems. In all reported cases, the recovered Hessian remains positive definite, every linearized matrix satisfies the Schur-complement nonsingularity diagnostic, and the bad-row and positive-off-diagonal ratios decrease under refinement. The condition-number estimates again exhibit $O(h^{-2})$-type growth. These results provide consistent numerical evidence that the localized algebraic diagnostics remain informative for the Newton systems arising in this test.

\begin{table}[H]
\centering
\footnotesize
\setlength{\tabcolsep}{3pt}
\caption{Worst-case Newton diagnostics on unstructured meshes for Example~3.}
\label{tab:ma_nonuniform_diagnostics}
\begin{tabular}{cccccccc}
\toprule
$n_I$ & iter.
& $\min \lambda_{\min}(H_hu_h^n)$
& $\operatorname{condest}$
& $\rho_{\rm bad}$
& $\rho_+$
& $\operatorname{condest}(\mathfrak S_h)$
& Schur \\
\midrule
433   & 4 & $6.3425{\times}10^{-1}$ & $5.3984{\times}10^2$ & $12.01\%$ & $2.46\%$ & $3.1284{\times}10^1$ & yes \\
1809  & 4 & $6.3725{\times}10^{-1}$ & $2.5270{\times}10^3$ & $9.29\%$  & $1.80\%$ & $1.2423{\times}10^2$ & yes \\
7393  & 3 & $6.3822{\times}10^{-1}$ & $1.0275{\times}10^4$ & $5.48\%$  & $1.04\%$ & $3.0584{\times}10^2$ & yes \\
29889 & 3 & $6.3868{\times}10^{-1}$ & $4.1435{\times}10^4$ & $2.95\%$  & $0.56\%$ & $6.6360{\times}10^2$ & yes \\
\bottomrule
\end{tabular}
\end{table}

Table~\ref{tab:ma_nonuniform_errors} reports the solution errors and Newton iteration counts. The number of Newton steps remains essentially unchanged under refinement, with only three or four iterations required on the reported meshes. The solution errors follow the same pattern as in the linear $P_1$ tests, with approximately second-order convergence in $e_0$ and first-order convergence in $e_1$, while the recovered Hessian error decreases steadily. Together with the diagnostics in Table~\ref{tab:ma_nonuniform_diagnostics}, these results indicate that the recovered-Hessian construction can be incorporated into a Newton linearization for this Monge--Amp\`ere type test.

\begin{table}[H]
\centering
\footnotesize
\setlength{\tabcolsep}{3pt}
\caption{Errors and Newton iterations on unstructured meshes for Example~3.}
\label{tab:ma_nonuniform_errors}
\begin{tabular}{cccccccc}
\toprule
\multirow{2}{*}{$n_I$}
& \multicolumn{2}{c}{$e_0$}
& \multicolumn{2}{c}{$e_1$}
& \multicolumn{2}{c}{$e_H$}
& \multirow{2}{*}{iter.} \\
\cmidrule(lr){2-3}\cmidrule(lr){4-5}\cmidrule(lr){6-7}
& error & order & error & order & error & order & \\
\midrule
433 & $1.30{\times}10^{-3}$ & -- & $2.72{\times}10^{-2}$ & -- & $2.75{\times}10^{-2}$ & -- & 4 \\
1809 & $2.46{\times}10^{-4}$ & 2.40 & $1.36{\times}10^{-2}$ & 1.00 & $1.25{\times}10^{-2}$ & 1.14 & 4 \\
7393 & $8.10{\times}10^{-5}$ & 1.60 & $6.80{\times}10^{-3}$ & 1.00 & $5.90{\times}10^{-3}$ & 1.08 & 4 \\
29889 & $2.02{\times}10^{-5}$ & 2.00 & $3.40{\times}10^{-3}$ & 1.00 & $2.90{\times}10^{-3}$ & 1.02 & 3 \\
120193 & $5.07{\times}10^{-6}$ & 2.00 & $1.70{\times}10^{-3}$ & 1.00 & $1.40{\times}10^{-3}$ & 1.05 & 3 \\
\bottomrule
\end{tabular}
\end{table}

\section{Conclusion}\label{sec:conclusion}

This paper presents a Hessian-recovery-based \(C^0\) finite element framework for elliptic equations in non-divergence form. The method replaces the strong Hessian \(D^2u\) by the recovered Hessian \(H_hu_h\) and builds the recovered residual \(A:H_hu_h-f\). This residual leads to three complementary discretizations: a nodal recovered-residual scheme, a Galerkin-type recovered-residual formulation, and a biorthogonal Petrov--Galerkin realization of the nodal formulation for linear elements.

A central issue for this class of \(C^0\) nonvariational discretizations is that the usual variational stability arguments do not apply directly. We therefore studied the recovered nodal operator through its matrix structure. The row-sum identity inherited from Hessian recovery leads to two verifiable algebraic solvability mechanisms: a globally monotone regime based on a discrete maximum principle and a Schur-complement criterion for localized sign violations. We also obtained a uniform inverse bound, an \(O(h^{-2})\) conditioning estimate in the monotone case, and residual consistency estimates based on the Hessian recovery error. Combined with the uniform inverse bound, the nodal consistency estimate gives a nodal \(L^\infty\)-error estimate for Scheme~1 in the globally monotone regime.

The numerical experiments provide consistent evidence for the accuracy of the recovered-residual discretizations and for the relevance of the proposed algebraic diagnostics. In the nonsmooth and discontinuous coefficient tests, the reported row-sum, sign-violation, condition-number, and Schur-complement diagnostics reflect the two algebraic regimes identified in the analysis. The Monge--Amp\`ere type experiment further illustrates that the recovered-Hessian construction can be used within a Newton linearization in which the coefficient matrices are generated dynamically from the current recovered Hessian. These observations motivate further work on sharper stability and error estimates, particularly for the Galerkin-type formulation, whose mass-averaged matrix structure is not covered by the nodal monotonicity analysis, and for fully nonlinear extensions in which the coefficient matrices are generated dynamically during the nonlinear iteration.

\end{document}